\documentclass[11pt]{article}
\usepackage{amsmath,amsthm,amssymb,amscd,fancybox,ifthen}
\usepackage{times}

\newcommand\dbarb{\bar{\partial}_b}
\newcommand\boxb{\square_b}
\DeclareMathOperator\eo{eo}
\DeclareMathOperator\ooee{oe}
\newcommand\dVol{\operatorname{dV}}

\newcommand\Dom{\operatorname{Dom}}

\newcommand\spnc{\operatorname{Spin}_{\bbC}}

\newcommand\hn{(1,-\frac 12)}
\newcommand\cT{\mathcal T}
\newcommand\cH{\mathcal H}
\newcommand\cI{\mathcal I}

\newcommand\cO{\mathcal O}
\newcommand\cQ{\mathcal Q}

\newcommand\cM{\mathcal M}

\newcommand\cB{\mathcal B}
\newcommand\cR{\mathcal R}
\newcommand\cP{\mathcal P}
\newcommand\cV{\mathcal V}
\newcommand\bzero{\boldsymbol 0}

\newcommand\bX{\overline{X}}

\newcommand\dbar{\bar{\pa}}
\newcommand\dbnc{\dbar\rho\rfloor}

\newcommand\bcQ{\bar{\mathcal Q}}
\newcommand\starb{\star_b}
\newcommand\talpha{\widetilde{\alpha}}
\newcommand\tbeta{\widetilde{\beta}}
\newcommand\tgamma{\widetilde{\gamma}}
\newcommand\ta{\widetilde{a}}
\newcommand\tb{\tilde{b}}
\newcommand\txi{\widetilde{\xi}}

\newcommand{\Spn}{S\mspace{-10mu}/ }

\newcommand\Ker{\operatorname{ker}}

\newcommand\cS{\mathcal{S}}
\newcommand\hcS{\widehat{\mathcal{S}}}
\newcommand\bcS{\bar{\mathcal{S}}}
\newcommand\cD{\mathcal{D}}

\newcommand\nsigma{\sigma_\nu}
\newcommand\bsigma{\sigma_b}

\newcommand\bZ{\overline{Z}}

\newcommand\ha{\frac12}

\renewcommand\Im{\operatorname{Im}}

\newcommand\bbC{\mathbb C}

\newcommand\pa{\partial}

\newcommand\restrictedto{\upharpoonright}

\newcommand\CI{{\mathcal C}^{\infty}}

\newcommand\Id{\operatorname{Id}}

\DeclareMathOperator{\odd}{o}
\DeclareMathOperator{\even}{e}
\DeclareMathOperator{\Ind}{Ind}

\DeclareMathOperator{\Rind}{R-Ind}
\DeclareMathOperator{\ind}{Ind}

\newtheorem{theorem}{Theorem}
\newtheorem{proposition}{Proposition}
\newtheorem{corollary}{Corollary}
\newtheorem{lemma}{Lemma}

\theoremstyle{definition}

\theoremstyle{remark}
\newtheorem{remark}{Remark}

\begin{document}

\title{Subelliptic Spin\,${}_{\bbC}$ Dirac operators, I} 

\author{Charles L. Epstein\footnote{ Research partially supported by NSF grants DMS99-70487 and
DMS02-03795, and the Francis J. Carey term chair.  E-mail: cle@math.upenn.edu}
\\ Department of Mathematics\\ University of Pennsylvania}

\date{May 23, 2005: Revised version}

\maketitle

\centerline{\Large{\it Dedicated to my parents, Jean and Herbert Epstein,}}
\centerline{\Large{\it on the occasion of their eightieth birthdays}}

\begin{abstract} Let $X$ be a compact K\"ahler manifold with strictly
pseudoconvex boundary, $Y.$ In this setting, the Spin${}_{\bbC}$ Dirac operator
is canonically identified with
$\dbar+\dbar^*:\CI(X;\Lambda^{0,\even})\rightarrow\CI(X;\Lambda^{0,\odd}).$ We
consider modifications of the classical $\dbar$-Neumann conditions that define
Fredholm problems for the Spin${}_{\bbC}$ Dirac operator. In part
2,~\cite{Epstein3}, we use boundary layer methods to obtain subelliptic
estimates for these boundary value problems. Using these results, we obtain an
expression for the finite part of the holomorphic Euler characteristic of a
strictly pseudoconvex manifold as the index of a Spin${}_{\bbC}$-Dirac operator
with a subelliptic boundary condition. We also prove an analogue of the
Agranovich-Dynin formula expressing the change in the index in terms of a
relative index on the boundary. If $X$ is a complex manifold partitioned by a
strictly pseudoconvex hypersurface, then we obtain formul\ae\ for the
holomorphic Euler characteristic of $X$ as sums of indices of
Spin${}_{\bbC}$-Dirac operators on the components. This is a subelliptic
analogue of Bojarski's formula in the elliptic case.
\end{abstract}

\section*{Introduction}
Let $X$ be an even dimensional manifold with a Spin${}_{\bbC}$-structure,
 see~\cite{duistermaat,LawsonMichelsohn}.  A compatible choice of metric, $g,$ defines a
 Spin${}_{\bbC}$-Dirac operator, $\eth$ which acts on sections of the bundle of
 complex spinors, $\Spn.$ The metric on $X$ induces a metric on the bundle of
 spinors. If $\langle\sigma,\sigma\rangle_g$ denotes a pointwise inner
 product, then we define an inner product of the space of sections of
 $\Spn,$  by setting:
$$\langle\sigma,\sigma\rangle_X=\int\limits_{X}\langle\sigma,\sigma\rangle_g
dV_g$$

If $X$ has an almost complex structure, then this structure defines a
Spin${}_{\bbC}$-structure.  If the complex structure is integrable,
then the bundle of complex spinors is canonically identified with
$\oplus_{q\geq 0}\Lambda^{0,q}.$ As we usually work with the chiral operator,
we let
\begin{equation}
\Lambda^{\even}=\bigoplus\limits_{q=0}^{\lfloor\frac{n}{2}\rfloor}\Lambda^{0,2q}\quad
\Lambda^{\odd}=\bigoplus\limits_{q=0}^{\lfloor\frac{n-1}{2}\rfloor}\Lambda^{0,2q+1}.
\end{equation}
If the metric is K\"ahler, then the Spin${}_{\bbC}$ Dirac operator is given by
$$\eth=\dbar+\dbar^*.$$ 
Here $\dbar^*$ denotes the formal adjoint of $\dbar$
defined by the metric. This operator is called the Dolbeault-Dirac operator by
Duistermaat, see~\cite{duistermaat}.  If the metric is Hermitian, though not
K\"ahler, then
\begin{equation}
\eth=\dbar+\dbar^*+\cM_0,
\label{7.28.1}
\end{equation}
here $\cM_0$ is a homomorphism carrying $\Lambda^{\even}$ to
$\Lambda^{\odd}$ and vice versa. It vanishes at points where the metric is
K\"ahler.  It is customary to write $\eth=\eth^{\even}+\eth^{\odd}$
where
$$\eth^{\even}:\CI(X;\Lambda^{\even})\longrightarrow
\CI(X,\Lambda^{\odd})$$
and $\eth^{\odd}$ is the formal adjoint of $\eth^{\even}.$ If $X$ is a
compact, complex manifold, then the graph closure of $\eth^{\even}$ is a Fredholm
operator. It has the same principal symbol as $\dbar+\dbar^*$  and therefore
its index is given by  
\begin{equation}
\ind(\eth^{\even})=\sum_{j=0}^n(-1)^j\dim H^{0,j}(X)=\chi_{\cO}(X).
\label{1.22.03.1}
\end{equation}

If $X$ is a manifold with boundary, then the kernels and cokernels of
$\eth^{\eo}$ are generally infinite dimensional. To obtain a Fredholm operator
we need to impose boundary conditions. In this instance there are no local
boundary conditions for $\eth^{\eo}$ that define elliptic problems.  Starting
with Atiyah, Patodi and Singer, boundary conditions defined by classical
pseudodifferential projections have been the focus of most of the work in this
field. Such boundary conditions are very useful for studying topological
problems, but are not well suited to the analysis of problems connected to the
holomorphic structure of $X.$ To that end we begin the study of boundary
conditions for $\eth^{\eo}$ obtained by modifying the classical $\dbar$-Neumann
and dual $\dbar$-Neumann conditions.  For a $(0,q)$-form, $\sigma^{0q},$ The
$\dbar$-Neumann condition  is the requirement that
$$\dbnc[\sigma^{0q}]_{bX}=0.$$ This imposes no condition if $q=0,$ and all
square integrable holomorphic functions thereby belong to the domain of the
operator, and define elements of the null space of $\eth^{\even}.$ Let $\cS$
denote the Szeg\H o projector; this is an operator acting on functions on $bX$
with range equal to the null space of the tangential Cauchy-Riemann operator,
$\dbarb.$ We can remove the null space in degree $0$ by adding the condition
\begin{equation}
\cS[\sigma^{00}]_{bX}=0.
\end{equation}
This, in turn, changes the boundary condition in degree $1$ to
\begin{equation}
(\Id-\cS)[\dbnc\sigma^{01}]_{bX}=0.
\end{equation}
If $X$ is strictly pseudoconvex, then these modifications to the
$\dbar$-Neumann condition produce a Fredholm boundary value problem for $\eth.$
Indeed, it is not necessary to use the exact Szeg\H o projector, defined by the
induced CR-structure on $bX.$ Any generalized Szeg\H o projector, as defined
in~\cite{EpsteinMelrose3}, suffices to prove the necessary estimates. There are
analogous conditions for strictly pseudoconcave
manifolds. In~\cite{BaumDouglasTaylor} and~\cite{Taylor7, Taylor8} the Spin${}_{\bbC}$
Dirac operator with the $\dbar$-Neumann condition is considered, though from a very
different perspective. The results in these papers are largely orthogonal to
those we have obtained.

A pseudoconvex manifold is denoted by $X_+$ and objects associated with it are
labeled with a $+$ subscript, e.~g., the $\spnc$-Dirac operator on $X_+$ is
denoted $\eth_+.$ Similarly, a pseudoconcave manifold is denoted by $X_-$ and
objects associated with it are labeled with a $-$ subscript. Usually $X$
denotes a compact manifold, partitioned by an embedded, strictly pseudoconvex
hypersurface, $Y$ into two components, $X\setminus Y=X_+\coprod X_-.$ 

If $X_{\pm}$ is either strictly pseudoconvex or strictly pseudoconcave, then the
modified boundary conditions are subelliptic and define Fredholm operators. The
indices of these operators are connected to the holomorphic Euler
characteristics of these manifolds with boundary, with the contributions of the
infinite dimensional groups removed. We also consider the Dirac operator acting
on the twisted spinor bundles
$$\Lambda^{p,\eo}=\Lambda^{\eo}\otimes\Lambda^{p,0},$$ 
and more generally
$\Lambda^{\eo}\otimes\cV$ where $\cV\to X$ is a holomorphic vector bundle. When
necessary, we use $\eth^{\eo}_{\cV\pm}$ to specify the twisting bundle. The
boundary conditions are defined by projection operators $\cR^{\eo}_{\pm}$
acting on boundary values of sections of $\Lambda^{\eo}\otimes\cV.$ Among other
things we show that the index of $\eth^{\even}_+$ with boundary condition defined by
$\cR^{e}_+$ equals the regular part of the holomorphic Euler characteristic:
\begin{equation}
\Ind(\eth^{\even}_{+},\cR^{\even}_+)=
\sum_{q=1}^{n}\dim H^{0,q}(X) (-1)^q.
\end{equation}

In~\cite{Epstein3}
we show that the pairs $(\eth^{\eo}_{\pm},\cR^{\eo}_{\pm})$ are Fredholm and
identify their $L^2$-adjoints. In each case, the $L^2$-adjoint is the closure
of the formally adjoint boundary value problem, e.~g.
$$(\eth_+^{\even},\cR_+^{\even})^*=\overline{(\eth_+^{\odd},\cR_+^{\odd})}.$$
This is proved by using a boundary layer method to reduce to analysis of
operators on the boundary. The operators we obtain on the boundary are neither
classical, nor Heisenberg pseudodifferential operators, but rather operators
belonging to the extended Heisenberg calculus introduced
in~\cite{EpsteinMelrose3}. Similar classes of operators were also introduced by
Beals, Greiner and Stanton as well as Taylor,
see~\cite{Beals-Stanton,Beals-Greiner1,Taylor6}.  In this paper we apply the
analytic results obtained in~\cite{Epstein3} to obtain Hodge decompositions for
each of the boundary conditions and $(p,q)$-types.

In the Section~\ref{s.bkgrnd} we review some well known facts about the
$\dbar$-Neumann problem and analysis on strictly pseudoconvex CR-manifolds. In
the following two sections we introduce the boundary conditions we consider in
the remainder of the paper and deduce subelliptic estimates for these
boundary value problems from the results in~\cite{Epstein3}.  The fourth
section introduces the natural dual boundary conditions.  In
Section~\ref{s.highnorm} we deduce the Hodge decompositions associated to the
various boundary value problems defined in the earlier sections.  In
Section~\ref{s.nullsp} we identify the nullspaces of the various boundary value
problems when the classical Szeg\H o projectors are used. In the
Section~\ref{s.agrdyn} we establish the basic link between the boundary
conditions for $(p,q)$-forms considered in the earlier sections and boundary
conditions for $\eth^{\eo}_{\pm}$ and prove an analogue of the Agranovich-Dynin
formula. In Section~\ref{s.lexseq} we obtain ``regularized'' versions of some
long exact sequences due to Andreotti and Hill. Using these sequences we prove
gluing formul\ae\ for the holomorphic Euler characteristic of a compact complex
manifold, $X,$ with a strictly pseudoconvex separating hypersurface.  These
formul\ae\ are subelliptic analogues of Bojarski's gluing formula for the
classical Dirac operator with APS-type boundary conditions.

{\small
\centerline{Acknowledgments} Boundary conditions similar to those considered in this paper
were first suggested to me by Laszlo Lempert. I would like to thank John Roe
for some helpful pointers on the Spin${}_{\bbC}$ Dirac operator.}

\section{Some background material}\label{s.bkgrnd}

Henceforth $X_{+}$ ($X_-$) denotes a compact complex manifold of complex
dimension $n$ with a strictly pseudoconvex (pseudoconcave) boundary. We assume
that a Hermitian metric, $g$ is fixed on $X_{\pm}.$ For some of our results we
make additional assumptions on the nature of $g,$ e.~g., that it is
K\"ahler. This metric induces metrics on all the natural bundles defined by the
complex structure on $X_{\pm}.$ To the extent possible, we treat the two cases
in tandem. For example, we sometimes use $bX_{\pm}$ to denote the boundary of
either $X_{+}$ or $X_{-}.$ The kernels of $\eth_{\pm}$ are both infinite
dimensional.  Let $\cP_{\pm}$ denote the operators defined on $bX_{\pm}$ which
are the projections onto the boundary values of element in $\ker\eth_{\pm};$
these are the Calderon projections. They are classical pseudodifferential
operators of order 0; we use the definitions and analysis of these operators
presented in~\cite{BBW}. 

We often work with the chiral Dirac operators $\eth^{\eo}_{\pm}$ which act on
sections of
\begin{equation}
\Lambda^{p,\even}=\bigoplus\limits_{q=0}^{\lfloor \frac n2\rfloor}
\Lambda^{p,2q}X_{\pm}\quad
\Lambda^{p,\odd}=\bigoplus\limits_{q=0}^{\lfloor \frac {n-1}2\rfloor}
\Lambda^{p,2q+1}X_{\pm},
\end{equation}
respectively. Here $p$ is an integer between $0$ and $n;$ except when entirely
necessary it is omitted from the notation for things like $\cR^{\eo}_{\pm},
\eth^{\eo}_{\pm},$ etc. The $L^2$-closure of the operators $\eth^{\eo}_{\pm},$ with
domains consisting of smooth spinors such that
$\cP^{\eo}_{\pm}(\sigma\big|_{bX_{\pm}})=0,$ are elliptic operators with
Fredholm index zero.

Let $\rho$ be a smooth defining function for the boundary of $X_{\pm}.$ Usually
we take $\rho$ to be \emph{negative} on $X_+$ and \emph{positive} on $X_-,$ so
that $\pa\dbar\rho$ is positive definite near $bX_{\pm}.$ If $\sigma$ is a
section of $\Lambda^{p,q},$ smooth up to $bX_{\pm},$ then the $\dbar$-Neumann
boundary condition is the requirement that
\begin{equation}
\dbar\rho\rfloor\sigma\restrictedto_{bX_{\pm}}=0.
\label{dbrnc}
\end{equation} 
If $X_{+}$ is strictly pseudoconvex, then there is a constant $C$ such that if
$\sigma$ is a smooth section of $\Lambda^{p,q},$ with $q\geq 1,$
satisfying~\eqref{dbrnc}, then $\sigma$ satisfies the \emph{basic estimate}:
\begin{equation}
\|\sigma\|_{(1,-\frac 12)}^2\leq
C(\|\dbar\sigma\|_{L^2}^2+\|\dbar^*\sigma\|_{L^2}^2+
\|\sigma\|_{L^2}^2).
\label{bscest}
\end{equation}
If $X_{-}$ is strictly pseudoconcave, then there is a constant $C$ such that if
$\sigma$ is a smooth section of $\Lambda^{p,q},$ with $q\neq n-1,$
satisfying~\eqref{dbrnc}, then  $\sigma$ again satisfies the basic
estimate~\eqref{bscest}. The $\square$-operator is defined formally as
$$\square\sigma=(\dbar\dbar^*+\dbar^*\dbar)\sigma.$$ The $\square$-operator,
with the $\dbar$-Neumann boundary condition is the graph closure of $\square$
acting on smooth forms, $\sigma,$ that satisfy~\eqref{dbrnc}, such that
$\dbar\sigma$ also satisfies~\eqref{dbrnc}. It has an infinite dimensional
nullspace acting on sections of $\Lambda^{p,0}(X_+)$ and
$\Lambda^{p,n-1}(X_-),$ respectively. For clarity, we sometimes use the
notation $\square^{p,q}$ to denote the $\square$-operator acting on sections of
$\Lambda^{p,q}.$ 

Let $Y$ be a compact strictly pseudoconvex CR-manifold of real dimension
$2n-1.$ Let $T^{0,1}Y$ denote the $(0,1)$-part of $TY\otimes\bbC$ and $\cT Y$
the holomorphic vector bundle $TY\otimes\bbC/T^{0,1} Y.$ The dual bundles are
denoted $\Lambda^{0,1}_b$ and $\Lambda^{1,0}_b$ respectively. For $0\leq p\leq
n,$ let
\begin{equation}
\CI(Y;\Lambda_b^{p,0})
\overset{\dbarb}{\longrightarrow}\CI(Y;\Lambda_b^{p,1})
\overset{\dbarb}{\longrightarrow}\hdots\overset{\dbarb}{\longrightarrow}
\CI(Y;\Lambda_b^{p,n-1})
\label{dbrbc}
\end{equation}
denote the $\dbarb$-complex. Fixing a choice of Hermitian metric on $Y,$ we
define formal adjoints
$$\dbarb^*:\CI(Y;\Lambda_b^{p,q})\longrightarrow \CI(Y;\Lambda_b^{p,q-1}).$$
The $\boxb$-operator acting on $\Lambda_b^{p,q}$ is the graph closure of 
\begin{equation}
\boxb=\dbarb\dbarb^*+\dbarb^*\dbarb,
\end{equation}
acting on $\CI(Y;\Lambda_b^{p,q}).$ The operator $\boxb^{p,q}$ is subelliptic
if $0<q<n-1.$ If $q=0,$ then $\dbarb$ has an infinite dimensional nullspace,
while if $q=n-1,$ then $\dbarb^*$ has an infinite dimensional nullspace. We let
$\cS_p$ denote an orthogonal projector onto the nullspace of $\dbarb$ acting on
$\CI(Y;\Lambda_b^{p,0}),$ and $\bcS_p$ an orthogonal projector onto the
nullspace of $\dbarb^*$ acting on $\CI(Y;\Lambda_b^{p,n-1}).$ The operator
$\cS_p$ is usually called ``the'' Szeg\H o projector; we call $\bcS_p$ the
conjugate Szeg\H o projector. These projectors are only defined once a metric
is selected, but this ambiguity has no bearing on our results. As is well
known, these operators are \emph{not} classical pseudodifferential operators,
but belong to the Heisenberg calculus. Generalizations of these projectors are
introduced in~\cite{EpsteinMelrose3} and play a role in the definition of
subelliptic boundary value problems for $\eth.$ For $0<q<n-1,$ the Kohn-Rossi
cohomology groups
$$H_b^{p,q}(Y)=\frac{\Ker\{ \dbarb:\CI(Y;\Lambda_b^{p,q})
\to\CI(Y;\Lambda_b^{p,q+1})\}}
{\dbarb \CI(Y;\Lambda_b^{p,q-1})}$$
are finite dimensional. The regularized $\dbarb$-Euler characteristics of $Y$
are defined to be
\begin{equation}
\chi_{pb}'(Y)=\sum_{q=1}^{n-2}(-1)^q\dim H_b^{p,q}(Y),\text{ for }0\leq p\leq n.
\label{dbrbeul}
\end{equation}
Very often we use $Y$ to denote the boundary of $X_{\pm}.$

The Hodge star operator on $X_{\pm}$ defines an isomorphism
\begin{equation}
\star:\Lambda^{p,q}(X_{\pm})\longrightarrow \Lambda^{n-p,n-q}(X_{\pm}).
\label{hdgstr1}
\end{equation}
Note that we have incorporated  complex conjugation into the definition of
the Hodge star operator. The usual identities continue to hold, i.~e.,
\begin{equation}
\star\star=(-1)^{p+q},\quad\dbar^*=-\star\dbar\star.
\label{04.20}
\end{equation}
There is also a Hodge star operator on $Y$ that defines an isomorphism:
\begin{equation}
\starb:\Lambda^{p,q}_b(Y)\longrightarrow
\Lambda^{n-p,n-q-1}_b(Y),\quad[\dbarb^{p,q}]^*=(-1)^{p+q+1}\starb\dbarb\starb.
\label{eqn04.17}
\end{equation}

There is a canonical boundary condition dual to the $\dbar$-Neumann
condition. The dual $\dbar$-Neumann condition is the requirement that
\begin{equation}
\dbar\rho\wedge\sigma\restrictedto_{bX_{\pm}}=0.
\label{dldbrnc}
\end{equation} 
If $\sigma$ is a $(p,q)$-form defined on $X_{\pm},$ then, along the
boundary we can write
\begin{equation}
\sigma\restrictedto_{bX_{\pm}}
=\dbar\rho\wedge(\dbar\rho\rfloor\sigma)+\sigma_b.
\label{neqn2}
\end{equation}
Here $\sigma_b\in\CI(Y;\Lambda_b^{p,q})$ is a representative of
 $\sigma\restrictedto_{(\cT Y)^p\otimes (T^{0,1}Y)^q}.$ The dual $\dbar$-Neumann
 condition is equivalent to the condition
\begin{equation}
\sigma_b=0.
\label{dldbrnc2}
\end{equation}
For later applications we note the following well known relations: For sections
$\sigma\in\CI(\bX_{\pm},\Lambda^{p,q}),$ we have
\begin{equation}
(\dbar\rho\rfloor \sigma)^{\starb}=(\sigma^{\star})_b,\quad
\dbar\rho\rfloor(\sigma^\star)=\bsigma^{\starb},\quad
(\dbar\sigma)_b =\dbarb\sigma_b.
\label{strrel}
\end{equation}

The dual $\dbar$-Neumann operator on $\Lambda^{p,q}$ is the graph closure of
$\square^{p,q}$ on smooth sections, $\sigma$ of $\Lambda^{p,q}$
satisfying~\eqref{dldbrnc}, such that $\dbar^*\sigma$ also
satisfies~\eqref{dldbrnc}.  For a strictly pseudoconvex manifold, the basic
estimate holds for $(p,q)$-forms satisfying~\eqref{dldbrnc}, provided $0\leq
q\leq n-1.$ For a strictly pseudoconcave manifold, the basic estimate holds for
$(p,q)$-forms satisfying~\eqref{dldbrnc}, provided $q\neq 1.$

As we consider many different boundary conditions, it is useful to have
notations that specify the boundary condition under consideration. If $\cD$
denotes an operator acting on sections of a complex vector bundle, $E\to X$ and $\cB$
denotes a boundary operator acting on sections of $E\restrictedto_{bX},$ then
the pair $(\cD,\cB)$ is the operator $\cD$ acting on smooth sections $s$ that
satisfy
$$\cB s\restrictedto_{bX}=0.$$ 
The notation $s\restrictedto_{bX}$ refers to the section of
$E\restrictedto_{bX}$ obtained by restricting a section $s$ of $E\to X$ to the
boundary. The operator $\cB$ is a pseudodifferential operator acting on
sections of $E\restrictedto_{bX}.$ Some of the boundary conditions we consider
are defined by Heisenberg pseudodifferential operators. We often denote objects
connected to $(\cD,\cB)$ with a subscripted $\cB.$ For example, the nullspace of
$(\cD,\cB)$ (or harmonic sections) might be denoted $\cH_{\cB}.$ We denote
objects connected to the $\dbar$-Neumann operator with a subscripted $\dbar,$
e.~g., $\square^{p,q}_{\dbar}.$ Objects connected to the dual $\dbar$-Neumann
problem are denoted by a subscripted $\dbar^*,$ e.~g.,
$\square^{p,q}_{\dbar^*}.$

Let $\cH_{\dbar}^{p,q}(X_{\pm})$ denote the nullspace of
$\square^{p,q}_{\dbar}$ and $\cH_{\dbar^*}^{p,q}(X_{\pm})$ the nullspace of
$\square^{p,q}_{\dbar^*}.$ In~\cite{kohn-rossi} it is shown that
\begin{equation}
\begin{split}
\cH_{\dbar}^{p,q}(X_+)&\simeq [\cH_{\dbar^*}^{n-p,n-q}(X_+)]^*,\text{ if }q\neq 0,\\
\cH_{\dbar}^{p,q}(X_-)&\simeq [\cH_{\dbar^*}^{n-p,n-q}(X_-)]^*,\text{ if }q\neq n-1.
\end{split}
\label{duality1}
\end{equation}

\begin{remark} In this paper $C$ is used to denote a variety of \emph{positive}
constants which depend only on the geometry of $X.$ If $M$ is a manifold with a
volume form $\dVol$ and $f_1,f_2$ are sections of a bundle with a Hermitian metric
$\langle\cdot,\cdot\rangle_g,$ then the $L^2$-inner product over $M$ is denoted
by
\begin{equation}
\langle f_1,f_2\rangle_{M}=\int\limits_{M}\langle f_1,f_2\rangle_g \dVol.
\end{equation}
\end{remark}

\section{Subelliptic boundary conditions for pseudoconvex manifolds}\label{s.pscnvx}
In this section we define a modification of the classical $\dbar$-Neumann
condition for sections belonging to $\CI(\bar{X}_+;\Lambda^{p,q}),$ for $0\leq
p\leq n$ and $0\leq q\leq n.$ The bundles $\Lambda^{p,0}$ are holomorphic, and
so, as in the classical case they do not not really have any effect on the
estimates. As above, $\cS_p$ denotes an orthogonal projection acting on
sections of $\Lambda^{p,0}_b$ with range equal to the null space of $\dbarb$
acting  sections of $\Lambda^{p,0}_b.$ The range of $\cS_p$
includes the boundary values of holomorphic $(p,0)$-forms, but may in general
be somewhat larger.  If $\sigma^{p0}$ is a holomorphic section, then
$\sigma^{p0}_b=\cS_{p}\sigma^{p0}_b.$ On the other hand, if $\sigma^{p0}$ is
any smooth section of $\Lambda^{p,0},$ then $\dbar\rho\rfloor\sigma^{p0}=0$ and
therefore, the $L^2$-holomorphic sections belong to the nullspace of
$\square^{p0}_{\dbar}.$

To obtain a subelliptic boundary value problem for $\square^{pq}$ in all
degrees, we modify the $\dbar$-Neumann condition in degrees $0$ and $1.$ The
modified boundary condition is denoted by $\cR_+.$ A smooth form
$\sigma^{p0}\in\Dom(\dbar_{\cR_+}^{p,0})$ provided
\begin{equation}
\cS_p\sigma^{p0}_b=0.
\label{eqn04.1}
\end{equation}
There is no boundary condition if $q>0.$ A smooth form belongs to
$\Dom([\dbar_{\cR_+}^{p,q}]^*)$ provided
\begin{equation}
\begin{split}
(\Id-\cS_p)[\dbar\rho\rfloor\sigma^{p1}]_b&=0,\\
[\dbar\rho\rfloor\sigma^{pq}]_b &=0\quad\text{ if }1<q.
\end{split}
\label{eqn04.2}
\end{equation}
For each $(p,q)$ we define the quadratic form
\begin{equation}
\cQ^{p,q}(\sigma^{pq})=\langle\dbar\sigma^{pq},\dbar\sigma^{pq}\rangle_{L^2}+
\langle\dbar^*\sigma^{pq},\dbar^*\sigma^{pq}\rangle_{L^2}
\end{equation}

We can  consider more general conditions than these by replacing the
classical Szeg\H o projector $\cS_p$ by a generalized Szeg\H o projector acting
on sections of $\Lambda^{p,0}_b.$  
Recall that an order zero operator, $S_E$ in the Heisenberg calculus, acting on
sections of a complex vector bundle $E\to Y$ is a generalized Szeg\H o projector if
\begin{enumerate}
\item $S_E^2=S_E$ and $S_E^*=S_E.$
\item $\sigma^H_0(S_E)=s\otimes \Id_E$ where $s$ is the symbol of a field of
  vacuum state projectors defined by a choice of compatible almost complex
  structure on the contact field of $Y.$
\end{enumerate}
This class of projectors is defined in~\cite{EpsteinMelrose} and analyzed in detail
in~\cite{EpsteinMelrose3}. Among other things we show that, given a generalized Szeg\H o
projector, there is a $\dbarb$-like operator, $D_E$ so that the range of
$S_E$ is precisely the null space of $D_E.$ The operator $D_E$ is $\dbarb$-like
in the following sense: If ${\bZ'_j}$ is a local frame field for the almost
complex structure defined by the principal symbol of $S_E,$ then there are
order zero Heisenberg operators $\mu_j,$ so that, locally
\begin{equation}
D_E \sigma=0\text{ if and only if }(\bZ_j'+\mu_j)\sigma=0\text{ for
}j=1,\dots,n-1.
\label{eqn6.17.2}
\end{equation}
Similar remarks apply to define generalized conjugate Szeg\H o projectors. 
We use the notation $\cS'_p$ to denote a generalized Szeg\H o projector acting
on sections of $\Lambda^{p,0}_b.$ 

We can view these boundary conditions as boundary conditions for the
operator $\eth_+$ acting on sections of $\oplus_q\Lambda^{p,q}.$ Let $\sigma$
be a such a section.  The boundary condition is expressed as a projection
operator acting on $\sigma\restrictedto_{bX_+}.$ We write
\begin{equation}
\begin{split}
\sigma\restrictedto_{bX_+}=\bsigma&+\dbar\rho\wedge\nsigma,\text{ with }\\
\bsigma=(\sigma^{p0}_b,\tilde\bsigma^{p})&\text{ and }
\nsigma=(\nsigma^{p1},\tilde\nsigma^{p}).
\end{split}
\end{equation}
Recall that $\bsigma^{pn}$ and $\nsigma^{p0}$ always vanish. With this notation
we have, in block form, that
\begin{equation}
\cR_+'\sigma\restrictedto_{bX_+}=
\left(\begin{matrix}\begin{matrix} \cS_p' & 0 \\
0 & \bzero\\
\end{matrix} &
\begin{matrix} 0 & 0\\
0&\bzero
\end{matrix}\\
\begin{matrix} 0 & 0\\
0&\bzero
\end{matrix}&  
\begin{matrix}\Id- \cS_p' & 0 \\
0 & \Id\\
\end{matrix} \end{matrix}\right)
\left(\begin{matrix} \bsigma^{p0}\\
\tilde\bsigma^{p}\\
\nsigma^{p1}\\
\tilde\nsigma^{p}\end{matrix}\right)
\label{7.27.1}
\end{equation}
Here $\bzero$ denotes an $(n-1)\times (n-1)$ matrix of zeros. The boundary
condition for $\eth_+$ is $\cR_+'\sigma\restrictedto_{bX_+}=0.$ These can of
course be split into boundary conditions for $\eth_+^{\eo},$ which we denote by
$\cR_+^{\prime\eo}.$ The formal adjoint of $(\eth_+^{\even},\cR_+^{\prime\even})$ is
$(\eth_+^{\odd},\cR_+^{\prime\odd}).$ In Section~\ref{s.agrdyn} we show that the
$L^2$-adjoint of $(\eth_+^{\even},\cR_+^{\prime\even})$ is the graph closure of
$(\eth_+^{\odd},\cR_+^{\prime\odd}).$ When the distinction is important, we
explicitly indicate the dependence on $p$ by using $\cR_{p+}'$ to denote the
projector acting on sections of $\oplus_q\Lambda^{p,q}\restrictedto_{bX_+}$ and
$\eth_{p+}$ to denote the operator acting on sections of
$\oplus_q\Lambda^{p,q}.$ 

We use $\cR_+$ (without the ${}'$) to denote the boundary condition defined by
the matrix in~\eqref{7.27.1}, with $\cS_p'=\cS_p,$ the classical Szeg\H o
projector. In~\cite{Epstein3}, we prove estimates for the Spin${}_{\bbC}$-Dirac
operator with these sorts of boundary conditions.  We first state a direct
consequence of Corollary 13.9 in~\cite{BBW}.
\begin{lemma}\label{lem111} Let $X$ be a complex manifold with boundary and $\sigma^{pq}\in
  L^2(X;\Lambda^{p,q}).$ Suppose that $\dbar\sigma^{pq},\dbar^*\sigma^{pq}$ are
  also square integrable, then $\sigma^{pq}\restrictedto_{bX}$ is well defined
  as an element of $H^{-\ha}(bX; \Lambda^{p,q}_{bX}).$
\end{lemma}
\begin{proof} Because $X$ is a complex manifold, the twisted Spin${}_{\bbC}$-Dirac
  operator acting on sections of $\Lambda^{p,*}$ is given
  by~\eqref{7.28.1}. The hypotheses of the lemma therefore imply that
  $\eth\sigma^{pq}$ is square integrable and the lemma follows directly from
  Corollary 13.9 in~\cite{BBW}.
\end{proof}
\begin{remark} If the restriction of a section of a vector bundle to the
  boundary is well defined in the sense of distributions then we say that the
  section has distributional boundary values. Under the hypotheses of the
  Lemma, $\sigma^{pq}$ has distributional boundary values.
\end{remark}

Theorem 3 in~\cite{Epstein3} implies the following estimates for the
individual form degrees:
\begin{proposition} Suppose that $X$ is a strictly pseudoconvex manifold,
  $\cS_p'$ is a generalized Szeg\H o projector acting on sections of
  $\Lambda^{p,0}_b,$ and let $s\in [0,\infty).$ There is a constant $C_s$ such
  that if $\sigma^{pq}$ is an
  $L^2$-section of $\Lambda^{p,q}$ with $\dbar\sigma^{pq},
  \dbar^*\sigma^{pq}\in H^s$ and 
\begin{equation}
\begin{split}
&\cS_p'[\sigma^{pq}]_b=0\quad\text{ if } q=0\\
&(\Id-\cS_p')[\dbnc\sigma^{pq}]_b=0\quad\text{ if } q=1\\
&[\dbnc\sigma^{p1}]_b=0\quad\text{ if } q>1,
\end{split}
\label{7.27.4}
\end{equation}
then
\begin{equation}
\|\sigma^{pq}\|_{H^{s+\ha}}\leq
C_s[\|\dbar\sigma^{pq}\|_{H^s}+\|\dbar^*\sigma^{pq}\|_{H^s}+\|\sigma^{pq}\|_{L^2}]
\label{7.27.3}
\end{equation}
\end{proposition}
\begin{remark} As noted in~\cite{Epstein3}, the hypotheses of the proposition
  imply that $\sigma^{pq}$ has a well defined restriction to $bX_+$ as an
  $L^2$-section of $\Lambda^{pq}\restrictedto_{bX_+}.$ The boundary conditions
  in~\eqref{7.27.4} can therefore be interpreted in the sense of distributions. If
  $s=0$ then the norm on the left hand side of~\eqref{7.27.3} can be replaced
  by the slightly stronger $H_{\hn}$-norm.
\end{remark}
\begin{proof} These estimates follow immediately from Theorem 3
  in~\cite{Epstein3} by observing that the hypotheses imply that
\begin{equation}
\begin{split}
&\eth_{\Lambda^{p,0}+}\sigma^{pq}\in H^s(X_+)\text{ and }\\
&\cR^{\prime}_{\Lambda^{p,0}+}[\sigma^{pq}]_{bX_+}=0.
\end{split}
\end{equation}
\end{proof}

These estimates show that, for all $0\leq p,q\leq n,$ the form domain for
$\bcQ^{p,q}_{\cR_{+}},$ the closure of $\cQ^{p,q}_{\cR_{+}},$ lies in
$H_{\hn}(X_+;\Lambda^{p,q}).$ This implies that the self adjoint operator,
$\square^{p,q}_{\cR_+},$ defined by the Friedrichs extension process, has a
compact resolvent and therefore a finite dimensional null space
$\cH^{p,q}_{\cR_+}(X_+).$ We define closed, unbounded operators on
$L^2(X_+;\Lambda^{p,q})$ denoted $\dbar^{p,q}_{\cR_+}$ and $[
\dbar^{p,q-1}_{\cR_+}]^*$ as the graph closures of $\dbar$ and $\dbar^*$ acting
on smooth sections with domains given by the appropriate condition
in~\eqref{eqn04.1},~\eqref{eqn04.2}.  The domains of these operators are
denoted $\Dom_{L^2}(\dbar^{p,q}_{\cR_+}),
\Dom_{L^2}([\dbar^{p,q-1}_{\cR_+}]^*),$ respectively. It is clear that
$$\Dom(\bcQ^{p,q}_{\cR_{+}})=\Dom_{L^2}(\dbar^{p,q}_{\cR_+})\cap
\Dom_{L^2}([\dbar^{p,q-1}_{\cR_+}]^*).$$

\section{Subelliptic boundary conditions for pseudoconcave manifolds}\label{s.pscncv}

We now repeat the considerations of the previous section for $X_-,$ a strictly
pseudoconcave manifold. In this case the $\dbar$-Neumann condition fails to
define a subelliptic boundary value problem on sections of $\Lambda^{p,n-1}.$
We let $\bcS_p$ denote an orthogonal projection onto the nullspace of
$[\dbarb^{p(n-1)}]^*.$ The projector acts on sections of
$\Lambda_b^{p(n-1)}.$ From this observation, and equation~\eqref{eqn04.17}, it
follows immediately that
\begin{equation}
\bcS_p=\starb\cS_{n-p}\starb.
\label{04.5}
\end{equation}
If instead we let $\cS_{n-p}'$ denote a generalized Szeg\H o projector acting
on $(n-p,0)$-forms, then~\eqref{04.5}, with $\cS_{n-p}$ replaced by
$\cS_{n-p}',$ defines a generalized conjugate Szeg\H o projector acting on
$(p,n-1)$-forms, $\bcS'_p.$

Recall that the defining function, $\rho,$ is positive on the interior of
$X_-.$ We now define a modified $\dbar$-Neumann condition for $X_-,$ which we
denote by $\cR_-^{\prime}.$ The $\Dom(\dbar^{p,q}_{\cR_-^{\prime}})$ requires
no boundary condition for $q\neq n-1$ and is specified for $q=n-1$ by
\begin{equation}
\bcS_p'\bsigma^{p(n-1)}=0.
\label{04.6}
\end{equation}
The $\Dom([\dbar^{p,q}_{\cR_-^{\prime}}]^*)$ is given by
\begin{eqnarray}
\dbar\rho\rfloor\sigma^{pq}&=&0\quad\text{ if } q\neq n\label{04.71}\\
(\Id-\bcS_p')(\dbar\rho\rfloor\sigma^{pn})_b&=&0
\label{04.7}
\end{eqnarray}

As before we assemble the individual boundary conditions into a boundary
condition for $\eth_-.$ The boundary condition is expressed as a projection
operator acting on $\sigma\restrictedto_{bX_-}.$ We write
\begin{equation}
\begin{split}
\sigma\restrictedto_{bX_-}=\bsigma&+\dbar\rho\wedge\nsigma,\text{ with }\\
\bsigma=(\tilde\bsigma^{p},\bsigma^{p(n-1)})&\text{ and }
\nsigma=(\tilde\nsigma^{p},\nsigma^{pn}).
\end{split}
\end{equation}
Recall that $\bsigma^{pn}$ and $\nsigma^{p0}$ always vanish. With this notation
we have, in block form that
\begin{equation}
\cR_-^{\prime}\sigma\restrictedto_{bX_-}=
\left(\begin{matrix}\begin{matrix}\bzero & 0 \\
0 &  \bcS_p' \\
\end{matrix} &
\begin{matrix} 0 & 0\\
0&\bzero
\end{matrix}\\
\begin{matrix} 0 & 0\\
0&\bzero
\end{matrix}&  
\begin{matrix}\Id & 0 \\
0 & \Id-\bcS_p'\\
\end{matrix} \end{matrix}\right)
\left(\begin{matrix} \tilde\bsigma^{p}\\
\bsigma^{p(n-1)}\\
\tilde\nsigma^{p}\\
\nsigma^{pn}\end{matrix}\right)
\label{7.27.6}
\end{equation}
Here $\bzero$ denotes an $(n-1)\times (n-1)$ matrix of zeros. The boundary
condition for $\eth_-$ is $\cR_-'\sigma\restrictedto_{bX_-}=0.$ These can of
course be split into boundary conditions for $\eth_-^{\eo},$ which we denote by
$\cR_-^{\prime\eo}.$ The formal adjoint of $(\eth_-^{\even},\cR_-^{\prime\even})$ is
$(\eth_-^{\odd},\cR_-^{\prime\odd}).$ In Section~\ref{s.agrdyn} we show that the
$L^2$-adjoint of $(\eth_-^{\eo},\cR_-^{\prime\eo})$ is the graph closure of
$(\eth_-^{\ooee},\cR_-^{\prime\ooee}).$ When the distinction is important, we
explicitly indicate the dependence on $p$ by using $\cR_{p-}'$ to denote this
projector acting on sections of $\oplus_q\Lambda^{p,q}\restrictedto_{bX_-}$ and
$\eth_{p-}$ to denote the operator acting on sections of
$\oplus_q\Lambda^{p,q}.$ If we are using the classical conjugate Szeg\H o
projector, then we omit the prime, i.e., the notation $\cR_{-}$  refers to the
boundary condition defined by the matrix in~\eqref{7.27.6} with $\bcS_p'=\bcS_p,$ the
classical conjugate Szeg\H o projector.

Theorem 3 in~\cite{Epstein3} also provides subelliptic estimates in this case.
\begin{proposition} Suppose that $X$ is a strictly pseudoconcave manifold,
  $\bcS_p'$ is a generalized Szeg\H o projector acting on sections of
  $\Lambda^{p,n-1}_b,$ and let $s\in [0,\infty).$ There is a constant $C_s$ such
  that if $\sigma^{pq}$ is an
  $L^2$-section of $\Lambda^{p,q}$ with $\dbar\sigma^{pq},
  \dbar^*\sigma^{pq}\in H^s$ and 
\begin{equation}
\begin{split}
&\bcS_p'[\sigma^{pq}]_b=0\quad\text{ if } q=n-1\\
&(\Id-\bcS_p')[\dbnc\sigma^{pq}]_b=0\quad\text{ if } q=n\\
&[\dbnc\sigma^{pq}]_{bX_-}=0\quad\text{ if } q\neq n-1, n,
\end{split}
\label{7.27.7}
\end{equation}
then
\begin{equation}
\|\sigma^{pq}\|_{H^{s+\ha}}\leq
C_s[\|\dbar\sigma^{pq}\|_{H^s}+\|\dbar^*\sigma^{pq}\|_{H^s}+\|\sigma^{pq}\|_{L^2}]
\label{7.27.8}
\end{equation}
\end{proposition}
\begin{proof} The hypotheses imply that
\begin{equation}
\begin{split}
&\eth_{\Lambda^{p,0}-}\sigma^{pq}\in H^s(X_-)\text{ and }\\
&\cR^{\prime}_{\Lambda^{p,0}-}[\sigma^{pq}]_{bX_-}=0.
\end{split}
\end{equation}
Thus $\sigma^{pq}$ satisfies the hypotheses of Theorem 3 in~\cite{Epstein3}.
\end{proof}

\section{The dual boundary conditions}\label{s.dualbc}
In the two previous sections we have established the basic estimates for $L^2$
forms on $X_+$ (resp. $X_-$) that satisfy $\cR_+'$ (resp. $\cR_-'$). The
Hodge star operator defines isomorphisms
\begin{equation}
\star: L^2(X_{\pm};\oplus_q\Lambda^{p,q})\longrightarrow L^2(X_{\pm};\oplus_q
\Lambda^{n-p,n-q}).
\end{equation}
Under this isomorphism, a form satisfying
$\cR_{\pm}'\sigma\restrictedto_{bX_{\pm}}=0$ is carried to a form,
${}^{\star}\sigma,$ satisfying
$(\Id-\cR_{\mp}')\star\sigma\restrictedto_{bX_{\pm}}=0,$ and vice versa. Here
of course the generalized Szeg\H o and conjugate Szeg\H o projectors must be
related as in~\eqref{04.5}. In form degrees where $\cR_{\pm}'$ coincides with
the usual $\dbar$-Neumann conditions, this statement is proved
in~\cite{FollandKohn1}. In the degrees where the boundary condition has been
modified, it follows from the identities in~\eqref{strrel}
and~\eqref{04.5}. Applying Hodge star, we immediately deduce the basic
estimates for the dual boundary conditions, $\Id-\cR_{\mp}'.$
\begin{lemma} Suppose that $X_{+}$ is strictly pseudoconvex and
  $\sigma^{pq}\in L^2(X_+;\Lambda^{p,q}).$ For $s\in [0,\infty),$ there is a
  constant $C_s$ so that, if $\dbar\sigma^{pq},\dbar^*\sigma^{pq}\in
  H^s,$ and 
\begin{equation}
\begin{split}
&\sigma^{pq}_b=0\quad\text{ if }q<n-1\\
&(\Id-\bcS_{p}')\sigma^{pq}_b=0\quad\text{ if }q=n-1\\
&\bcS_p'(\dbar\rho\rfloor\sigma^{pq})_b=0\quad\text{ if }q=n,
\end{split}
\end{equation}
then
\begin{equation}
\|\sigma^{pq}\|_{H^{s+\ha}}\leq
C_s\left[\|\dbar\sigma^{pq}\|_{H^s}+\|\dbar^*\sigma^{pq}\|_{H^s}+
\|\sigma^{pq}\|_{L^2}^2\right].
\end{equation}
\end{lemma}

\begin{lemma} Suppose that $X_{-}$ is strictly pseudoconcave and
  $\sigma^{pq}\in L^2(X_-;\Lambda^{p,q}).$  For $s\in [0,\infty),$ there is a
  constant $C_s$ so that, if $\dbar\sigma^{pq},\dbar^*\sigma^{pq}\in
  H^s,$ and
\begin{equation}
\begin{split}
&\sigma^{pq}_b=0\quad\text{ if }q>1\\
&\cS_{p}'(\dbnc\sigma^{pq})_b=0\text{ and }\sigma^{pq}_b=0\quad\text{ if }q=1\\
&(\Id-\cS_p')\sigma^{pq}_b=0\quad\text{ if }q=0,
\end{split}
\end{equation}
then
\begin{equation}
\|\sigma^{pq}\|_{H^{s+\ha}}\leq
C_s\left[\|\dbar\sigma^{pq}\|_{H^s}+\|\dbar^*\sigma^{pq}\|_{H^s}+
\|\sigma^{pq}\|_{L^2}^2\right].
\end{equation}
\end{lemma}

\section{Hodge decompositions}\label{s.highnorm}
The basic analytic ingredient that is needed to proceed is the higher norm
estimates for the $\square$-operator. Because the boundary conditions
$\cR_{\pm}'$ are nonlocal, the standard elliptic regularization and
approximation arguments employed, e.g., by Folland and Kohn do not directly
apply. Instead of trying to adapt these results and treat each degree $(p,q)$
separately, we instead consider the operators $\eth_{\pm}^{\eo}$ with boundary
conditions defined by $\cR^{\prime\eo}_{\pm}.$ In~\cite{Epstein3} we use a boundary layer
technique to obtain estimates for the inverses of the operators
$[\eth_{\pm}^{\eo}]^*\eth_{\pm}^{\eo}+\mu^2.$ On a K\"ahler manifold the
operators $[\eth_{\pm}^{\eo}]^*\eth_{\pm}^{\eo}$ preserve form degree, which
leads to estimates for the inverses of $\square^{p,q}_{\cR_{\pm}}+\mu^2.$ For
our purposes the following consequence of Corollary 3 in~\cite{Epstein3}
suffices.
\begin{theorem}\label{thm1} Suppose that $X_{\pm}$ is a strictly pseudoconvex
  (pseudoconcave) compact, complex K\"ahler manifold with boundary. Fix
  $\mu>0,$ and $s\geq 0.$ There is a positive constant $C_s$ such that for
  $\beta\in H^s(X_{\pm};\Lambda^{p,q}),$ there exists a unique section
  $\alpha\in H^{s+1}(X_{\pm};\Lambda^{p,q})$ satisfying
  $[\square^{p,q}+\mu^2]\alpha=\beta$ with
\begin{equation}
\alpha\in\Dom(\dbar^{p,q}_{\cR_{\pm}'})\cap\Dom([\dbar^{p,q-1}_{\cR_{\pm}'}]^*)\text{
  and }
\dbar\alpha\in\Dom([\dbar^{p,q}_{\cR_{\pm}'}]^*),\,
\dbar^*\alpha\Dom(\dbar^{p,q-1}_{\cR_{\pm}'})
\label{eqn04.21}
\end{equation}
such that
\begin{equation}
\|\alpha\|_{H^{s+1}}\leq C_s\|\beta\|_{H^s}
\end{equation}
\end{theorem}
The boundary conditions in~\eqref{eqn04.21} are in the sense of
distributions. If $s$ is sufficiently large, then we see that this boundary
value problem has a classical solution.

As in the classical case, these estimates imply that each operator
$\square^{p,q}_{\cR_{\pm}'}$ has a complete basis of eigenvectors composed of
smooth forms. Moreover the orthocomplement of the nullspace is the range.  This
implies that each operator has an associated Hodge decomposition. If
$G^{p,q}_{\cR_{\pm}'},\, H^{p,q}_{\cR_{\pm}'}$ are the partial inverse and
projector onto the nullspace, then we have that
\begin{equation}
\square^{p,q}_{\cR_{\pm}'}G^{p,q}_{\cR_{\pm}'}=
G^{p,q}_{\cR_{\pm}'}\square^{p,q}_{\cR_{\pm}'}=\Id-H^{p,q}_{\cR_{\pm}'}
\end{equation}

To get the usual and more useful Hodge decomposition, we use boundary conditions
defined by the classical Szeg\H o projectors. The basic property needed to
obtain these results is contained in the following two lemmas.
\begin{lemma}\label{lem10} If $\alpha\in\Dom_{L^2}(\dbar^{p,q}_{\cR_{\pm}}),$ then
  $\dbar\alpha\in\Dom_{L^2}(\dbar^{p,q+1}_{\cR_{\pm}}).$
\end{lemma}
\begin{proof} The $L^2$-domain of $\dbar^{p,q}_{\cR_{\pm}}$ is  defined as the
  graph closure of smooth forms satisfying the appropriate boundary
  conditions, defined by~\eqref{eqn04.1} and~\eqref{04.6}. Hence, if
  $\alpha\in\Dom_{L^2}(\dbar^{p,q}_{\cR_{\pm}}),$ then 
  there is a sequence of smooth $(p,q)$-forms $<\alpha_n>$ such that
\begin{equation}
\lim_{n\to\infty}\|\dbar\alpha_n-\dbar\alpha\|_{L^2}+\|\alpha_n-\alpha\|_{L^2}=0,
\end{equation}
and each $\alpha_n$ satisfies the appropriate boundary condition. 
First we consider $\cR_+.$ If $q=0,$ then
$\cS_p(\alpha_n)_b=0.$ The operator $\dbar^{p,1}_{\cR_{+}}$ has no boundary
condition, so $\dbar\alpha_n$ belongs to $\Dom(\dbar^{p,1}_{\cR_{+}}).$ Since
$\dbar^2\alpha_n=0.$ we see that
$\dbar\alpha\in\Dom_{L^2}(\dbar^{p,1}_{\cR_{+}}).$ In all other cases
$\dbar^{p,q}_{\cR_{+}}$ has no boundary condition. 

We now turn to $\cR_-.$ In this case there is only a boundary condition if
$q=n-1,$ so we only need to consider
$\alpha\in\Dom_{L^2}(\dbar^{p,n-2}_{\cR_{-}}).$ Let $<\alpha_n>$ be smooth
forms converging to $\alpha$ in the graph norm. Because $\bcS_p\dbarb=0,$
it follows that
$$\bcS_p(\dbar\alpha_n)_b=\bcS_p(\dbarb(\alpha_n)_b)=0.$$
Hence $\dbar\alpha_n\in\Dom(\dbar^{p,n-1}_{\cR_-}).$ Again $\dbar^2\alpha_n=0$
implies that $\dbar\alpha\in\Dom_{L^2}(\dbar^{p,n-1}_{\cR_-}).$
\end{proof}
\begin{remark} The same argument applies to show that the lemma holds for the
  boundary condition defined by   $\cR^{\prime}_+.$
\end{remark}

We have a similar result for the adjoint. The domains of
$[\dbar^{p,q}_{\cR_{\pm}}]^*$ are defined as the graph closures of
$[\dbar^{p,q}]^*$ with boundary conditions defined
by~\eqref{eqn04.2},~\eqref{04.71} and~\eqref{04.7}.
\begin{lemma}\label{lem11} 
If $\alpha\in\Dom_{L^2}([\dbar^{p,q}_{\cR_{\pm}}]^*)$ then
  $\dbar^*\alpha\in\Dom_{L^2}([\dbar^{p,q-1}_{\cR_{\pm}}]^*).$
\end{lemma}
\begin{proof} Let $\alpha\in\Dom_{L^2}([\dbar^{p,q}_{\cR_{\pm}}]^*).$ As before
  there is a sequence $<\alpha_n>$ of smooth forms in
  $\Dom([\dbar^{p,q}_{\cR_{\pm}}]^*),$ converging to $\alpha$ in the graph norm.
We need to consider the individual cases. We begin with $\cR_+.$  The only case
  that is not classical is that of $q=1.$ We suppose that $<\alpha_n>$ is a
  sequence of forms in $\CI(\bX_+;\Lambda^{p,2})$ with $\dbnc\alpha_n=0.$ Using
  the identities in~\eqref{strrel} we see that
\begin{equation}
[\dbnc\dbar^*\alpha_n]_b=[(\dbar^{\star}\alpha_n)_b]^{\starb}.
\end{equation}
On the other hand, as $(\dbnc\alpha_n)_b=0$ it follows that $({}^{\star}\alpha_n)_b=0$
and therefore 
$$(\dbar^{\star}\alpha_n)_b=\dbarb({}^{\star}\alpha_n)_b=0.$$
This shows that $(\Id-\cS_p)\dbnc\dbar^*\alpha_n=0$ and therefore
$\dbar^*\alpha_n$ is in the domain of $[\dbar^{p,0}_{\cR_+}]^*.$ As
$[\dbar^*]^2=0$ this shows that
$\dbar^*\alpha\in\Dom_{L^2}([\dbar^{p,0}_{\cR_{+}}]^*).$
 
On the pseudoconcave side we only need to consider $q=n-1.$ The boundary
condition implies that $\dbarb^*(\dbnc\alpha_n)_b=0.$ Using the identities
in~\eqref{strrel} we see that
\begin{equation}
\dbnc\dbar^*\alpha_n={}^{\starb}(\dbar^\star\alpha_n)_b=
\dbarb^*(\dbnc\alpha_n)_b=0.
\end{equation}
Thus $\dbar^*\alpha_n\in\Dom([\dbar^{p,n-2}_{\cR_-}]^*).$
\end{proof}
\begin{remark} Again, the same argument applies to show that the lemma holds for the
  boundary condition defined by   $\cR^{\prime}_+.$
\end{remark}

These lemmas show that, in the sense of closed operators, $\dbar_{\cR_{\pm}}^2$
and $[\dbar^*_{\cR_{\pm}}]^2$  vanish. This, along with the higher norm
estimates, give the strong form of the Hodge decomposition, as well as the
important commutativity results,~\eqref{7.28.3} and~\eqref{7.28.4}.
\begin{theorem}\label{thm2} Suppose that $X_{\pm}$ is a strictly pseudoconvex
  (pseudoconcave) compact, K\"ahler complex manifold with boundary.  For $0\leq
  p,q\leq n,$ we have the strong orthogonal decompositions
\begin{equation}
\alpha=\dbar\dbar^*G^{p,q}_{\cR_\pm}\alpha+\dbar^*\dbar
G^{p,q}_{\cR_\pm}\alpha+
H^{p,q}_{\cR_\pm}\alpha.
\end{equation}
If $\alpha\in\Dom_{L^2}(\dbar^{p,q}_{\cR_{\pm}})$ then
\begin{equation}
\dbar G^{p,q}_{\cR_{\pm}}\alpha=G^{p,q+1}_{\cR_{\pm}}\dbar\alpha.
\label{7.28.3}
\end{equation}
If $\alpha\in\Dom_{L^2}([\dbar^{p,q}_{\cR_{\pm}}]^*)$ then
\begin{equation}
\dbar^* G^{p,q}_{\cR_{\pm}}\alpha= G^{p,q-1}_{\cR_{\pm}}\dbar^*\alpha.
\label{7.28.4}
\end{equation}
\end{theorem}

Given Theorem~\ref{thm1} and Lemmas~\ref{lem10}--\ref{lem11} the proof of this
theorem is exactly the same as the proof of Theorem 3.1.14
in~\cite{FollandKohn1}. Similar decompositions also hold for the dual boundary
value problems defined by $\Id-\cR_+$ on $X_-$ and $\Id-\cR_-$ on $X_+.$ We
leave the explicit statements to the reader.

As in the case of the standard $\dbar$-Neumann problems these estimates show
that the domains of the self adjoint operators defined by the quadratic forms
$\cQ^{p,q}$ with form domains specified as the intersection of
$\Dom(\dbar^{p,q}_{\cR_\pm})\cap \Dom([\dbar^{p,q-1}_{\cR_\pm}]^*)$ are exactly
as one would expect.  As in~\cite{FollandKohn1} one easily deduces the
following descriptions of the unbounded self adjoint operators
$\square^{p,q}_{\cR_{\pm}}.$
\begin{proposition} Suppose that $X_+$ is strictly pseudoconvex, then the
  operator $\square^{p,q}_{\cR_+}$ with domain specified by
\begin{equation}
\begin{split}
&\sigma^{pq}\in\Dom_{L^2}(\dbar^{p,q}_{\cR_+})\cap
\Dom_{L^2}([\dbar^{p,q-1}_{\cR_+}]^*)\text{ such that}\\
&\dbar^*\sigma^{pq}\in\Dom_{L^2}(\dbar^{p,q-1}_{\cR_+})\text{ and }
\dbar\sigma^{pq}\in Dom_{L^2}([\dbar^{p,q}_{\cR_+}]^*)
\end{split}
\label{04.17}
\end{equation}
is a self adjoint operator. It coincides with the Friedrichs extension defined
by $\cQ^{pq}$ with form domain given by the first condition in~\eqref{04.17}.
\end{proposition}

\begin{proposition} Suppose that $X_-$ is strictly pseudoconcave, then the
  operator $\square^{p,q}_{\cR_-}$ with domain specified by
\begin{equation}
\begin{split}
&\sigma^{pq}\in\Dom_{L^2}(\dbar^{p,q}_{\cR_-})\cap
\Dom_{L^2}([\dbar^{p,q-1}_{\cR_-}]^*)\text{ such that}\\
&\dbar^*\sigma^{pq}\in\Dom_{L^2}(\dbar^{p,q-1}_{\cR_-})\text{ and }
\dbar\sigma^{pq}\in Dom_{L^2}([\dbar^{p,q}_{\cR_-}]^*)
\end{split}
\label{04.18}
\end{equation}
is a self adjoint operator. It coincides with the Friedrichs extension defined
by $\cQ^{pq}$ with form domain given by the first condition in~\eqref{04.18}.
\end{proposition}

\section{The nullspaces of the modified $\dbar$-Neumann problems}\label{s.nullsp}

As noted above $\square^{p,q}_{\cR_{\pm}}$ has a compact resolvent in all form
degrees and therefore the harmonic spaces $\cH^{p,q}_{\cR_{\pm}}(X_{\pm})$ are
finite dimensional. The boundary conditions easily imply that
\begin{eqnarray}
\cH^{p,0}_{\cR_+}(X_+)&=&0\text{ for all }p\text{ and }
\cH^{p,q}_{\cR_+}(X_+)=\cH^{p,q}_{\dbar}(X_+)\text{ for }q>1.
\label{04.19}\\
\cH^{p,q}_{\cR_-}(X_-)&=&\cH^{p,q}_{\dbar}(X_-)\text{ for }q<n-1.
\end{eqnarray}
We now identify $\cH^{p,1}_{\cR_+}(X_+),$ and $\cH^{p,n}_{\cR_-}(X_-),$ but
 leave $\cH^{p,n-1}_{\cR_-}(X_-)$ to the next section.
 
We begin with the pseudoconvex case.  To identify the null space of
 $\square^{p,1}_{\cR_+}$ we need to define the following vector space:
\begin{equation}
E^{p,1}_0(\bX_+)=\frac{\{\dbar\alpha:\: \alpha\in\CI(\bX_+;\Lambda^{p,0})
\text{ and }\dbar_b\alpha_b=0\}}
{\{\dbar\alpha:\: \alpha\in\CI(\bX_+;\Lambda^{p,0})
\text{ and }\alpha_b=0\}}.
\end{equation}
It is clear that $E^{p,1}_0(\bX_+)$ is a subspace of the ``zero''-cohomology group
$H^{p,1}_0(\bX_+)\simeq
\cH^{p,1}_{\dbar^*}(X_+)\simeq[\cH^{n-p,n-1}_{\dbar}]^*(X_+)$
and is therefore finite dimensional. If $X_+$ is a Stein manifold, then this
vector space is trivial.  It is also not difficult to show that
\begin{equation}
E^{p,1}_0(\bX_+)\simeq\frac{H^{p,0}(Y)}{[H^{p,0}(\bX_+)]_b}.
\label{eq06.1}
\end{equation}
Thus $E^{p,1}_0$ measures the extent of the failure of closed $(p,0)$ forms on
$bX_+$ to have  holomorphic extensions to $X_+.$ 

\begin{lemma}\label{lem17}
If $X_+$ is strictly pseudoconvex, then
$$\cH^{p,1}_{\cR_+}(X_+)\simeq\cH^{p,1}_{\dbar}(X_+)\oplus E^{p,1}_0.$$
\end{lemma}
\begin{proof} Clearly $\cH^{p,1}_{\cR_+}(X_+)\supset \cH^{p,1}_{\dbar}(X_+).$
If $\sigma^{p1}\in\cH^{p,1}_{\cR_+}(X_+),$ then
$$(\Id-\cS_p)(\dbnc\sigma^{p1})_b=0.$$ 
Let $\beta\in \cH^{p,0}_{\dbar}(X_+),$ then
\begin{equation}
0=\langle\dbar\beta,\sigma^{p1}\rangle_{X_+}=
\langle\beta,\dbnc\sigma^{p1}\rangle_{bX_+}
\end{equation}
Thus, we see that $\dbnc\sigma^{p1}$ is orthogonal to
$\cH^{p,0}_{\dbar}(X_+)\restrictedto_{bX_+}.$ 

Let $a\in \Im\cS_p\ominus \cH^{p,0}_{\dbar}(X_+)\restrictedto_{bX_+}.$ We now
show that there is an element $\alpha\in\cH^{p,1}_{\cR_+}(X_+)$ with
$\dbnc\alpha=a.$ Let $\ta$ denote a smooth extension of $a$ to $X_+.$ If
$\xi\in\cH^{p,0}_{\dbar}(X_+),$ then
\begin{equation}
\langle\dbar^*\dbar(\rho\ta),\xi\rangle_{X_+}=\langle a,\xi\rangle_{bX_+}.
\end{equation}
By assumption, $a$ is orthogonal to
$\cH^{p,0}_{\dbar}(X_+)\restrictedto_{bX_+},$ thus
$H^{p,0}_{\dbar}(\dbar^*\dbar(\rho\ta))=0.$ With
$b=G^{p,0}_{\dbar}\dbar^*\dbar(\rho\ta),$ we see that
\begin{equation}
\begin{split}
\dbar^*\dbar b= (\Id-H^{p,0}_{\dbar})&\dbar^*\dbar a=\dbar^*\dbar a\\
\dbnc\dbar b&=0.
\end{split}
\end{equation}
Hence if $\alpha=\dbar(\rho\ta-b),$ then $\dbar\alpha=\dbar^*\alpha=0,$ and
$\dbnc\alpha=a.$ If $\alpha_1, \alpha_2\in\cH^{p,1}_{\cR_+}(X_+)$ both satisfy
$\dbnc\alpha_1=\dbnc\alpha_2=a,$ then
$\alpha_1-\alpha_2\in\cH^{p,1}_{\dbar}(X_+).$ Together with the existence
result, this shows that
\begin{equation}
\frac{\cH^{p,1}_{\cR_+}(X_+)}{\cH^{p,1}_{\dbar}(X_+)}\simeq E^{p,1}_0,
\end{equation}
which completes the proof of the lemma.
\end{proof}

For the pseudoconcave side we have
\begin{lemma} If $X_-$ is strictly pseudoconcave then
$\cH^{pn}_{\cR_-}(X_-)\simeq[H^{n-p,0}(X_-)]^\star\simeq\cH^{p,n}_{\Id-\cR_+}(X_-).$
 \end{lemma}
\begin{proof} A $(p,n)$-form $\sigma^{pn}$ belongs to $\cH^{pn}_{\cR_-}(X_-)$ provided
   that
$$\dbar^*\sigma^{pn}=0,\text{ and }(\Id-\bcS_p)(\dbar\rho\rfloor\sigma^{pn})_b=0.$$
  The identities in~\eqref{04.20}  imply that ${}^\star\sigma^{pn}\in
  H^{n-p,0}(X_-).$ 

On the other hand, if $\eta\in H^{n-p,0}(X_-),$ then $\dbar^* {}^\star\eta=0,$
and $(\Id-\cS_{n-p})\eta_b=0.$ The identities in~\eqref{strrel}
and~\eqref{04.5} imply that $(\Id-\bcS_{p})(\dbnc{}^{\star}\eta)_b=0.$ This shows
that ${}^\star\eta\in\cH^{pn}_{\cR_-}(X_-),$ completing the proof of the first
isomorphism. A form $\eta\in\cH^{p,n}_{\Id-\cR_+}(X_-)$ provided that
$\dbar^*\eta=0.$ The boundary condition $\eta_b=0$ is vacuous for a
$(p,n)$-form. This shows that ${}^{\star}\eta\in H^{n-p,0}(X_-),$ the converse
is immediate. 
\end{proof}

All that remains is $\cH^{p,n-1}_{\cR_-}(X_-).$ This space does not have as
simple a description as the others. We return to this question in the next
section. We finish this section with the observation that the results
in Section~\eqref{s.dualbc} imply the following duality
statements, for $0\leq q,p\leq n:$ 
\begin{equation}
[\cH^{p,q}_{\cR_+}(X_+)]^* \simeq \cH^{n-p,n-q}_{\Id-\cR_-}(X_+),\quad
[\cH^{p,q}_{\cR_-}(X_-)]^* \simeq \cH^{n-p,n-q}_{\Id-\cR_+}(X_-).
\end{equation}
The isomorphisms are realized by applying the Hodge star operator.

\section{Connection to $\eth_{\pm}$ and the Agranovich-Dynin formula}\label{s.agrdyn}
Thus far we have largely considered one $(p,q)$-type at a time. As noted in the
introduction, by grouping together the even,  or odd forms we obtain bundles of
complex spinors on which the Spin${}_{\bbC}$-Dirac operator acts. We let
\begin{equation}
\Lambda^{p,\even}=\bigoplus\limits_{q=0}^{\lfloor \frac{n}{2}
  \rfloor}\Lambda^{p,2q},\quad
\Lambda^{p,\odd}=\bigoplus\limits_{q=0}^{\lfloor \frac{n-1}{2} \rfloor}\Lambda^{p,2q+1}.
\end{equation}
The bundles $\Lambda^{p,\even},\Lambda^{p,\odd}$ are the basic complex spinor
bundles, $\Lambda^{\even},\Lambda^{\odd},$ twisted with the holomorphic vector
bundles $\Lambda^{p,0}.$ Unless it is needed for clarity, we do not include the
value of $p$ in the notation.

Assuming that the underlying manifold is a K\"ahler manifold, the
Spin${}_{\bbC}$-Dirac operator is $\eth=\dbar+\dbar^*.$ It maps even forms to
odd forms and we denote by
\begin{equation}
\eth^{\even}_{\pm}:\CI(X_{\pm};\Lambda^{p,\even})\longrightarrow\CI(X_{\pm};\Lambda^{p,\odd}),\,
\eth^{\odd}_{\pm}:\CI(X_{\pm};\Lambda^{p,\odd})\longrightarrow\CI(X_{\pm};\Lambda^{p,\even}).
\end{equation}
As noted above, the boundary projection operators $\cR_{\pm}$ (or $\cR_{\pm}'$)
can be divided into operators acting separately on even and odd forms,
$\cR_{\pm}^{\eo},$ ( $\cR_{\pm}^{\prime\eo}$). These boundary conditions
define subelliptic boundary value problems for $\eth_{\pm}^{\eo}$ that are
closely connected to the individual $(p,q)$-types.  The connection is via the
basic integration-by-parts formul\ae\ for $\eth_{\pm}^{\eo}.$ There are several
cases, which we present in a series of lemmas.
\begin{lemma}\label{lembv1} If $\sigma\in\CI(\bX_{\pm};\Lambda^{p,\eo})$  satisfies
    $\cR^{\prime\eo}_{+}\sigma\restrictedto_{bX_{\pm}}=0$ or
$(\Id-\cR^{\prime\eo}_{-})\sigma\restrictedto_{bX_{\pm}}=0,$  then
\begin{equation}
\langle\eth_{\pm}\sigma,\eth_{\pm}\sigma\rangle_{X_{\pm}}=
\langle\dbar\sigma,\dbar\sigma\rangle_{X_{\pm}}+
\langle\dbar^*\sigma,\dbar^*\sigma\rangle_{X_{\pm}}
\label{1.16.03.1}
\end{equation}
\end{lemma}
\begin{remark} Note that when using the boundary conditions defined by $\cR_+$
  and $\Id-\cR_-,$ we are able to use a generalized Szeg\H o projector,
  unconnected to the complex structure on $X_{\pm}.$ This is not always true
  for $\cR_-$ and $\Id-\cR_+.$ See Lemmas~\ref{lembv2} and~\ref{lembv3}.
\end{remark}
\begin{proof}
The proof for $\cR_{\pm}^{\prime\eo}$ is a consequence of the facts that
\begin{enumerate}
\item[(a)] $\dbar^2=0$
\item[(b)] If $\eta$ is a $(p,j)$-form satisfying
  $\dbar\rho\rfloor\eta\restrictedto_{bX_{\pm}}=0,$ then, for $\beta$ any\newline
 smooth  $(p,j-1)$-form we have
\begin{equation}
\langle\beta,\dbar^*\eta\rangle_{X_{\pm}}=
\langle\dbar\beta,\eta\rangle_{X_{\pm}}.
\label{1.16.03.2}
\end{equation}
\end{enumerate}
We need to show that
\begin{equation}
\langle\dbar\sigma^{pq},\dbar^*\sigma^{p(q+2)}\rangle_{X_{\pm}}=0.
\label{eqn6.24.1}
\end{equation}
This follows immediately from (a), (b), and the fact that $\sigma^{p(q+2)}$
satisfies 
$$\dbnc\sigma^{p(q+2)}=0,\text{ for all }q\geq 0.$$

In the proof for $\Id-\cR^{\prime\eo}_{-},$ we replace (a) and (b) above with
\begin{enumerate}
\item[(a${}^{\prime}$)] $[\dbar^*]^2=0$
\item[(b${}^{\prime}$)] If $\eta$ is a $(p,j)$-form satisfying
  $\dbar\rho\wedge\eta\restrictedto_{bX_{\pm}}=0,$ then, for $\beta$ any\newline
  smooth  $(p,j+1)$-form we have
\begin{equation}
\langle\beta,\dbar\eta\rangle_{X_{\pm}}=
\langle\dbar^*\beta,\eta\rangle_{X_{\pm}}.
\label{1.16.03.22}
\end{equation}
\end{enumerate}
Since $(\Id-\cR^{\prime\eo}_{-})\sigma\restrictedto_{bX_{\pm}}=0$ implies that 
$\dbar\rho\wedge\sigma^{pq}\restrictedto_{bX_{\pm}}=0,$
holds for $q<n-1,$ the relation in~\eqref{eqn6.24.1} holds for all $q$ of
interest. This case could also be treated by observing that it is dual to $\cR_+'.$
\end{proof}

Now we consider $\cR_-$ and $\Id-\cR_+.$ Let $b_n$ denote the parity (even or odd) of $n,$
  and $\tb_n$ the opposite parity.
\begin{lemma}\label{lembv2} 
  If a section $\sigma\in\CI(\bX_{\pm};\Lambda^{p,\odd})$ satisfies
  $(\Id-\cR^{\prime\odd}_{+})\sigma\restrictedto_{bX_{\pm}}=0,$ or 
$\sigma\in\CI(\bX_{\pm};\Lambda^{p,\tb_n})$ satisfies
  $\cR^{\prime \tb_n}_{-}\sigma\restrictedto_{bX_{\pm}}=0,$
  then~\eqref{1.16.03.1} holds. 
\end{lemma}
\begin{remark} In these cases we can again use generalized Szeg\H o projectors.
\end{remark}
\begin{proof} The proofs here are very much as before. For $\Id-\cR_+^{\prime\odd}$ we use
  the fact that
\begin{equation}
\langle\dbar\sigma^{pq},\dbar^*\sigma^{p(q+2)}\rangle_{X_{\pm}}=
\langle\dbar\rho\wedge\sigma^{pq},\dbar^*\sigma^{p(q+2)}\rangle_{bX_{\pm}},
\end{equation}
and this vanishes if $q\geq 1.$ For $\cR_-^{\prime\tb_n}$ we use
  the fact that
\begin{equation}
\langle\dbar\sigma^{pq},\dbar^*\sigma^{p(q+2)}\rangle_{X_{\pm}}=
\langle\dbar\rho\wedge\sigma^{pq},\dbnc\sigma^{p(q+2)}\rangle_{bX_{\pm}},
\end{equation}
and this vanishes if $q<n-2.$
\end{proof}

In the final cases we are restricted to the boundary conditions which employ
the classical Szeg\H o projector defined by the complex structure on $X_{\pm}.$
\begin{lemma}\label{lembv3} 
  If a section $\sigma\in\CI(\bX_{\pm};\Lambda^{p,\even})$ satisfies
  $(\Id-\cR^{\even}_{+})\sigma\restrictedto_{bX_{\pm}}=0,$ or 
$\sigma\in\CI(\bX_{\pm};\Lambda^{p,b_n})$ satisfies
  $\cR^{b_n}_{-}\sigma\restrictedto_{bX_{\pm}}=0,$ then~\eqref{1.16.03.1}
  holds. 
\end{lemma}
\begin{proof}
First we consider  $\Id-\cR^{\even}_{+}.$ For even $q\geq 2,$ the proof
given above shows that~\eqref{eqn6.24.1} holds; so we are
left to consider $q=0.$ The boundary condition satisfied by $\sigma^{p0}$ is
$(\Id-\cS_p)\sigma^{p0}_b=0.$ Hence, we have
\begin{equation}
\begin{split}
\langle\dbar\sigma^{p0},\dbar^*\sigma^{p2}\rangle_{X_{\pm}}&=
\langle \dbar\sigma^{p0}_b,\dbnc\sigma^{p2}\rangle_{bX_{\pm}}\\
&=\langle \dbar\rho\wedge\dbar\sigma^{p0}_b,\sigma^{p2}\rangle_{bX_{\pm}}=0.
\end{split}
\end{equation}
The last equality follows because $\dbar\rho\wedge\dbar\sigma^{p0}=0$ if
$\dbarb\sigma^{p0}_b=0.$ 

Finally we consider $\cR_-.$  The proof given above suffices for $q<n.$ We need
to consider $q=n;$ in this case $(\Id-\bcS_p)(\dbnc\sigma^{pn})_b=0.$ We begin
by observing that
\begin{equation}
\begin{split}
\langle\dbar\sigma^{p(n-2)},\dbar^*\sigma^{pn}\rangle_{X_{\pm}}&=
\langle \dbarb\sigma^{p(n-2)}_b,(\dbnc\sigma^{pn})_b\rangle_{bX_{\pm}}\\
&=\langle \sigma^{p(n-2)}_b,\dbarb^*(\dbnc\sigma^{pn})_b\rangle_{bX_{\pm}}=0.
\end{split}
\end{equation}
The last equality follows from fact that$(\dbnc\sigma^{pn})_b=\bcS_p(\dbnc\sigma^{pn})_b.$
\end{proof}

In all cases where~\eqref{1.16.03.1} holds we can identify the null spaces of
the operators $\eth^{\eo}_{\pm}.$ Here we stick to the pseudoconvex side and
boundary conditions defined by the classical Szeg\H o projectors.  It
follows from~\eqref{1.16.03.1} that
\begin{equation}
\begin{split}
\Ker(\eth^{\even}_{p+},\cR^{\even}_+)&
=\bigoplus_{j=1}^{\lfloor \frac n2\rfloor}\cH_{\dbar}^{p,2j}(X_+),\\
\Ker(\eth^{\odd}_{p+},\cR^{\odd}_+)&=
E^{p,1}_0\oplus \bigoplus_{j=1}^{\lfloor \frac
  {n-1}2\rfloor}\cH_{\dbar}^{p,2j+1}(X_+)
\end{split}
\label{1.16.03.10}
\end{equation}
In~\cite{Epstein3} we identify the $L^2$-adjoints of the operators
$(\eth^{\eo}_{\pm},\cR^{\prime\eo}_{\pm}) $with the graph closures of the
formal adjoints, e.g,
\begin{equation}
\begin{split}
(\eth_{+}^{\eo},\cR_{+}^{\prime\eo})^*&=
\overline{(\eth_{+}^{\ooee},\cR_{+}^{\prime\ooee})}\\
(\eth_{-}^{\eo},\cR_{-}^{\prime\eo})^*&=
\overline{(\eth_{-}^{\ooee},\cR_{-}^{\prime\ooee})}.
\end{split}
\end{equation}
Using these identities, the Dolbeault isomorphism and standard facts about the
$\dbar$-Neumann problem on a strictly pseudoconvex domain, we obtain that
\begin{equation}
\Ind(\eth^{\even}_{p+},\cR^{\even}_+)=-\dim E^{p,1}_0+\sum_{q=1}^n(-1)^q\dim
H^{p,q}(X_+).
\end{equation}

Recall that if $\cS_p'$ and $\cS_p''$ are generalized Szeg\H o projectors, then
their relative index $\Rind(\cS_p',\cS_p'')$ is defined to be the Fredholm
index of the restriction
\begin{equation}
\cS_p'':\Im\cS_p'\longrightarrow\Im\cS_p''.
\end{equation}
For the pseudoconvex side we now prove an Agranovich-Dynin type formula.
\begin{theorem}\label{thm3.3}
Let $X_+$ be a compact strictly pseudoconvex K\"ahler manifold, with $\cS_p$
the classical Szeg\H o projector, defined as the projector onto the null space
of $\dbarb$ acting on $\CI(bX_+;\Lambda^{p,0}_b).$ If $\cS_p'$ is a generalized
Szeg\H o projector,  then
\begin{equation}
\Ind(\eth^{\even}_+,\cR^{\prime\even}_+)-
\Ind(\eth^{\even}_+,\cR^{\even}_+)=\Rind(\cS_p,\cS_p').
\end{equation}
\end{theorem}
\begin{proof} It follows from Lemma~\ref{lembv1} that all other groups are the
  same, so we only need to compare
  $\cH^{p,0}_{\cR^{\prime}_+}(X_+)$ to $\cH^{p,0}_{\cR^+}(X_+)$ and
  $\cH^{p,1}_{\cR^{\prime}_+}(X_+)$ to $\cH^{p,1}_{\cR^+}(X_+).$ For this
  purpose we introduce the subprojector $\hcS_p$ of $\cS_p,$ defined to be the
  orthogonal projection onto $\cH^{p,0}_{\dbar}(X_+)\restrictedto_{bX_+}.$ Note
  that
\begin{equation}
\Rind(\cS_p,\hcS_p)=\dim E^{p,1}_0.
\label{6.25.5}
\end{equation}

The $q=0$ case is quite easy. The group
  $\cH^{p,0}_{\cR^+}(X_+)=0.$ A section
  $\sigma^{p0}\in\cH^{p,0}_{\cR^{\prime}_+}(X_+),$ if and only if
  $\dbar\sigma^{p0}=0$ and $\cS_p'\sigma^{p0}_b=0.$ The first condition
  implies that $\sigma^{p0}_b\in\Im\hcS_p.$ Conversely, if
  $\eta\in\Ker[\cS_p':\Im\hcS_p\to\Im\cS_p'],$ then there is a unique
  holomorphic $(p,0)$-form $\sigma^{p0}$ with $\sigma^{p0}_b=\eta.$ This shows
  that
\begin{equation}
\cH^{p,0}_{\cR^{\prime}_+}(X_+)\simeq \Ker[\cS_p':\Im\hcS_p\to\Im\cS_p'].
\label{6.25.4}
\end{equation}

Now we turn to the $q=1$ case. No matter which boundary projection is used
\begin{equation}
\cH^{p,1}_{\dbar}(X_+)\subset\cH^{p,1}_{\cR^{\prime}_+}(X_+).
\label{6.25.1}
\end{equation}
As shown in Lemma~\ref{lem17}
\begin{equation}
\frac{\cH^{p,1}_{\cR_+}(X_+)}{\cH^{p,1}_{\dbar}(X_+)}\simeq E^{p,1}_0.
\label{6.25.3}
\end{equation} 
Now suppose that $\sigma^{p1}\in\cH^{p,1}_{\cR^{\prime}_+}(X_+)$ and
$\eta\in\cH^{p,0}_{\dbar}(X_+),$ then
\begin{equation}
0=\langle\dbar\eta,\sigma^{p1}\rangle_{X_+}=\langle\eta,(\dbnc\sigma^{p1})_b\rangle_{bX_+}.
\end{equation}
Hence $(\dbnc\sigma^{p1})_b\in\Ker[\hcS_p:\Im\cS_p'\to\Im\hcS_p].$

To complete the proof we need to show that for
$\eta_b\in\Ker[\hcS_p:\Im\cS_p'\to\Im\hcS_p]$ there is a harmonic $(p,1)$-form,
$\sigma^{p1}$ with $(\dbnc\sigma^{p1})_b=\eta_b.$ Let $\eta$ denote a smooth
extension of $\eta_b$ to $X_+.$ We need to show that there is a $(p,0)$ form
$\beta$ such that
\begin{equation}
\dbar^*\dbar(\rho\eta)=\dbar^*\dbar\beta\text{ and }(\dbnc\dbar\beta)_b=0.
\end{equation}
This follows from the fact that $\hcS_p\eta_b=0,$ exactly as in the proof of
Lemma~\ref{lem17}. Hence $\sigma^{p1}=\dbar(\rho\eta-\beta)$ is an element of
$\cH^{p,1}_{\cR^{\prime}_+}(X_+)$ such that $(\dbnc\sigma^{p1})_b=\eta_b.$ This
shows that
\begin{equation}
\frac{\cH^{p,1}_{\cR^{\prime}_+}(X_+)}{\cH^{p,1}_{\dbar}(X_+)}\simeq 
\Ker[\hcS_p:\Im\cS_p'\to\Im\hcS_p].
\label{6.25.2}
\end{equation}
Combining~\eqref{6.25.3} with~\eqref{6.25.2} we obtain that
\begin{equation}
\dim\cH^{p,1}_{\cR^{\prime}_+}(X_+)-\dim\cH^{p,1}_{\cR_+}(X_+)=
\dim \Ker[\hcS_p:\Im\cS_p'\to\Im\hcS_p]-\dim E^{p,1}_0.
\end{equation}
Combining this with~\eqref{6.25.4} and~\eqref{6.25.5} gives
\begin{equation}
\Ind(\eth_+^{\even},\cR^{\prime}_+)-\Ind(\eth_+^{\even},\cR_+)=
\Rind(\hcS_p,\cS_p')+\Rind(\cS_p,\hcS_p)=\Rind(\cS_p,\cS_p').
\end{equation}
The last equality follows from the cocycle formula for the relative index.
\end{proof}

\section{Long exact sequences and gluing formul\ae}\label{s.lexseq}

Suppose that $X$ is a compact complex manifold with a separating strictly
pseudoconvex hypersurface $Y.$ Let $X\setminus Y=X_+\coprod X_-,$ with $X_+$
strictly pseudoconvex and $X_-$ strictly pseudoconcave. A principal goal of
this paper is to express 
$$\chi^p_{\cO}(X)=\sum_{q=0}^n (-1)^q\dim H^{p,q}(X),$$ 
in terms of indices of operators on $X_{\pm}.$ Such results are classical for
topological Euler characteristic and Dirac operators with elliptic boundary
conditions, see for example Chapter 24 of~\cite{BBW}. In this section we modify
long exact sequences given by Andreotti and Hill in order to prove such results
for subelliptic boundary conditions.

The Andreotti-Hill sequences relate the smooth cohomology groups
 $$H^{p,q}(\bX_{\pm},\cI),\quad H^{p,q}(\bX_{\pm}),\quad \text{and
 }H^{p,q}_b(Y).$$ 
The notation $\bX_{\pm}$ is intended to remind the reader that these are
 cohomology groups defined by the $\dbar$-operator acting on forms that are
 smooth on the closed manifolds with boundary, $\bX_{\pm}.$ The differential
 ideal $\cI$ is composed of forms, $\sigma,$ so that near $Y,$ we have
\begin{equation}
\sigma=\dbar\rho\wedge\alpha+\rho\beta.
\end{equation}
These are precisely the forms that satisfy the dual $\dbar$-Neumann
condition~\eqref{dldbrnc}. If $\xi$ is a form defined on all of $X,$ then we use
the shorthand notation
$$\xi_{\pm}\overset{d}{=}\xi\restrictedto_{X_{\pm}}.$$

For a strictly pseudoconvex manifold, it follows from the Hodge decomposition
and the results in Section~\ref{s.nullsp} that
\begin{equation}
\begin{split}
H^{p,q}(\bX_+)\simeq\cH^{p,q}_{\dbar}(X_+)&\text{ for }q\neq 0,\text{ and }\\
H^{p,q}(\bX_+)\simeq\cH^{p,q}_{\cR_+}(X_+)&\text{ for }q\neq 0,1,
\end{split}
\label{eq1}
\end{equation}
and for a strictly pseudoconcave manifold
\begin{equation}
\begin{split}
H^{p,q}(\bX_-)\simeq\cH^{p,q}_{\dbar}(X_-)=&\cH^{p,q}_{\cR_-}(X_-)\text{ for }q\neq
n-1,n
\text{ and }\\
[H^{n-p,0}(X_-)]^{\star}&=\cH_{\cR_-}^{p,n}(X_-).
\end{split}
\label{eq2}
\end{equation}
By duality we also have the isomorphisms
\begin{equation}
\begin{split}
H^{p,q}(\bX_+,\cI)\simeq\cH^{p,q}_{\dbar^*}(X_+)\text{ for
}q\neq n,\text{ and }\\
H^{p,q}(\bX_+,\cI)\simeq\cH^{p,q}_{\Id-\cR_-}(X_+)\text{ for
}q\neq n,n-1,
\end{split}
\label{eq3}
\end{equation}
and for a strictly pseudoconcave manifold
\begin{equation}
\begin{split}
H^{p,q}(\bX_-,\cI)\simeq\cH^{p,q}_{\dbar^*}(X_-)=&\cH^{p,q}_{\Id-\cR_+}(X_-)\text{ for
}q\neq 0,1\text{ and }
\\H^{p,0}(X_-)&=\cH_{\Id-\cR_+}^{p,0}(X_-).
\end{split} 
\label{eq4}
\end{equation}

We  recall the definitions of  various maps introduced in~\cite{AnHi}:
\begin{equation}
\begin{split}
&\alpha_q:H^{p,q}(X)\longrightarrow H^{p,q}(\bX_+)\oplus H^{p,q}(\bX_-)\\
&\beta_q:H^{p,q}(\bX_+)\oplus H^{p,q}(\bX_-) \longrightarrow H^{p,q}_b(Y)\\
&\gamma_q:H^{p,q}_b(Y)\longrightarrow H^{p,q+1}(X).
\end{split}
\end{equation}
The first two are  simple
\begin{equation}
\alpha_q(\sigma^{pq})\overset{d}{=}\sigma^{pq}\restrictedto_{\bX_+}\oplus
\sigma^{pq}\restrictedto_{\bX_-}\quad 
\beta_q(\sigma^{pq}_+,\sigma^{pq}_-)\overset{d}{=}[\sigma^{pq}_+-\sigma^{pq}_-]_b.
\end{equation}
To define $\gamma_q$ we recall the notion of \emph{distinguished representative}
defined in~\cite{AnHi}: If $\eta\in H^{p,q}_b(Y)$ then there is a $(p,q)$-form
$\xi$ defined on $X$ so that
\begin{enumerate}
\item $\xi_b$ represents $\eta$ in $H^{p,q}_b(Y).$
\item $\dbar\xi$ vanishes to infinite order along $Y.$
\end{enumerate}
The map $\gamma_q$ is defined in terms of a distinguished representative $\xi$
for $\eta$ by
\begin{equation}
\gamma_q(\eta)\overset{d}{=}\begin{cases}\dbar\xi &\text{ on }\bX_+\\
-\dbar\xi &\text{ on }\bX_-.
\end{cases}
\end{equation}
As $\dbar\xi$ vanishes to infinite order along $Y,$ this defines a smooth
form.

The map $\talpha_0:H^{p,0}(X)\to H^{p,0}(\bX_-)$ is defined by restriction. To
define $\tbeta_0: H^{p,0}(\bX_-)\to E^{p,1}_0(\bX_+),$ we extend $\xi\in
H^{p,0}(\bX_-)$ to a smooth form, $\txi$ on all of $X$ and set
\begin{equation}
\tbeta_0(\xi)=\dbar\txi\restrictedto_{\bX_+}.
\end{equation}
It is easy to see that $\tbeta_0(\xi)$ is a well defined element of the
quotient, $E^{p,1}_0(\bX_+).$ To define $\tgamma_0:E^{p,1}_0(\bX_+)\to
H^{p,1}(X)$ we observe that an element $[\xi]\in E^{p,1}_0(\bX_+)$ has a
representative, $\xi$ which vanishes on $bX_+.$ The class $\tgamma_0([\xi])$ is
defined by extending such a representative by zero to $X_-.$ As noted
in~\cite{AnHi}, one can in fact choose a representative so that $\xi$ vanishes
to infinite order along $bX_+.$

We can now state our modification to the Mayer-Vietoris sequence in Theorem 1
in~\cite{AnHi}.
\begin{theorem}\label{thm3}  
Let $X,X_+, X_-, Y$ be as above. Then the following sequence
 is exact
\begin{equation}
\begin{CD}
&&0 @>>>\\
 H^{p,0}(X) @>\talpha_0>> H^{p,0}(\bX_-) @>\tbeta_0>>  E^{p,1}_0(\bX_+) \\
  @>\tgamma_0>> H^{p,1}(X)@>\alpha_1>> H^{p,1}(\bX_+)\oplus H^{p,1}(\bX_-) \\
@>\beta_1>>   H^{p,1}_b(Y) @>\gamma_1>>\cdots\\
 @>\beta_{n-2}>>  H^{p,n-2}_b(Y) @>\gamma_{n-2}>>
H^{p,n-1}(X)\\@>r_+\oplus H^{p,n-1}_{\cR_-}>>
H^{p,n-1}(\bX_+)\oplus \cH^{p,n-1}_{\cR_-}(X_-)
@>>>\frac{H^{p,n-1}(X_+)}{K^{p,n-1}_+}\\
@>>> 0.
\end{CD}
\label{les1}
\end{equation} 
Here  $r_+$ denotes restriction to $X_+$ and 
\begin{equation}
K^{p,n-1}_+=\{\alpha\in H^{p,n-1}(\bX_+):\:
\int\limits_{Y}\xi\wedge\alpha_b=0\text{ for all }
\xi\in H^{n-p,0}(\bX_-)\}.
\end{equation}
The last nontrivial map in~\eqref{les1} is the canonical quotient by the
subspace $K^{p,n-1}_+\oplus \cH^{p,n-1}_{\cR_-}(X_-).$
\end{theorem}
\begin{remark} Note that if $p=0,$ then $E^{0,1}_0=0.$ This follows 
  from~\eqref{eq06.1} and the fact that, on a strictly pseudoconvex manifold,
  all CR-functions on the boundary extend as holomorphic functions.  The proof
  given below works for all $n\geq 2.$ If $n=2,$ then one skips in~\eqref{les1}
  from $H^{p,1}(X)$ to $H^{p,1}(\bX_+)\oplus\cH^{p,1}_{\cR_-}(X_-).$
\end{remark}
\begin{proof}   It is clear that $\talpha_0$ is injective as $H^{p,0}(X)$ consists of
holomorphic forms. We now establish exactness at $H^{p,0}(\bX_-).$ That
$\Im\talpha_0\subset\Ker\tbeta_0$ is clear. Now suppose that on $\bX_+$ we have
$\tbeta_0(\xi)=0,$ this means that
\begin{equation}
\dbar\txi\restrictedto_{\bX_+}=\dbar\theta\text{ where }\theta_b=0.
\label{eq06.2}
\end{equation}
This implies that $\txi_+-\theta$ defines a holomorphic extension of $\xi$ to
all of $X$ and therefore $\xi\in\Im\talpha_0.$ That
$\Im\tbeta_0\subset\Ker\tgamma_0$ is again clear. Suppose on the other hand
that $\tgamma_0(\xi)=0.$ This means that there is a $(p,0)$-form, $\beta,$
defined on all of $X$ so that $\dbar\beta=\xi$ on $X_+$ and $\dbar\beta=0$ on
$X_-.$ This shows that $\xi=\tbeta_0(\beta_-).$ 

It is once again clear that $\Im\tgamma_0\subset\Ker\alpha_1.$ If
$\alpha_1(\xi)=0,$ then there are forms $\beta_{\pm}$ so that
\begin{equation}
\dbar\beta_{\pm}=\xi_{\pm}
\end{equation}
Let $\tbeta$ be a smooth extension of $\beta_-$ to all of $X.$ The form
$\xi-\dbar\tbeta$ represents the same class in $H^{p,1}(X)$ as $\xi.$ Since
\begin{equation}
(\xi-\dbar\tbeta)\restrictedto_{X_-}=0\text{ and }
(\xi-\dbar\tbeta)\restrictedto_{X_+}=\dbar(\beta_+-\tbeta_-),
\end{equation}
we see that $\xi\in\Im\tgamma_0.$

Exactness through $H^{p,n-2}_b(Y)$ is proved in~\cite{AnHi}. We now show
exactness at $H^{p,n-1}(X).$  The $\dbar$-Neumann condition, satisfied by
elements of $\cH^{p,n-1}_{\cR_{-}}(X_{-}),$ implies that
$H^{p,n-1}_{\cR_{-}}(\dbar\alpha_{-})=0,$ that $r_+(\dbar\alpha_+)=0$ is
obvious. Hence 
$$\Im\gamma_{n-2}\subset\left[\Ker r_+\oplus H^{p,n-1}_{\cR_-}\right].$$ 
Now suppose that $\beta\in H^{p,n-1}(X)$ satisfies
$H^{p,n-1}_{\cR_{-}}\beta_{-}=0,\, r_+(\beta_+)=0.$ The second condition
implies that
\begin{equation}
\beta_+=\dbar\gamma_+.
\end{equation}
Let $\gamma_-$ denote a smooth extension of $\gamma_+$ to $X_-.$ Then
$\beta_--\dbar\gamma_-$ vanishes along $Y$ and therefore Theorem~\ref{thm2} gives
\begin{equation}
\beta_--\dbar\gamma_-=\dbar\dbar^*G^{p,n-1}_{\cR_-}(\beta_--\dbar\gamma_-)
=\dbar \chi_-.
\end{equation}
Putting these equations together, we have shown that
\begin{equation}
\beta_+=\dbar\gamma_+,\quad \beta_-=\dbar(\gamma_-+\chi_-).
\end{equation}
Andreotti and Hill show that this implies that $\beta\in\Im\gamma_{n-2},$ thus
establishing exactness at $H^{p,n-1}(X).$

To show exactness at $H^{p,n-1}(X_+)\oplus \cH^{p,n-1}_{\cR_-}(X_-)$
we need to show that 
\begin{equation}
\Im\left[r_+\oplus H^{p,n-1}_{\cR_-}\right]=K^{p,n-1}_+\oplus
\cH^{p,n-1}_{\cR_-}(X_-).
\label{eq5}
\end{equation}
Let $\alpha\in\cH^{p,n-1}_{\cR_-}(X_-),$ then $\dbar\alpha=\dbar^*\alpha=0$ and
$(\dbnc\alpha)_b=\bcS_p\alpha_b=0.$ The last condition implies that
$$\alpha_b=\dbarb\beta.$$
We can extend $\beta$ to $\beta_+$ on $X_+$ so that $\dbnc\dbar\beta_+=0.$
Defining
\begin{equation}
\talpha=\begin{cases} \alpha&\text{ on }X_-\\
\dbar\beta_+&\text{ on }X_+,
\end{cases}
\end{equation}
gives a $\dbar$-closed form that defines a class in $H^{p,n-1}(X).$ It is clear
that
$$ r_+(\talpha_+)=0\text{ and
}H^{p,n-1}_{\cR_-}(\talpha_-)=\alpha.$$
To finish the argument we only need to describe
$I^{p,n-1}_+=\{ r_+(\theta):\: \theta\in H^{p,n-1}(X)\}.$  If
$\alpha_+$   belongs to  $I^{p,n-1}_+,$ then evidently $\alpha_+$ has a closed
extension to $X_-,$ call it $\alpha_-.$ If $\xi\in H^{n-p,0}(X_-),$ then
\begin{equation}
0=\int\limits_{X_-}\dbar(\alpha_-\wedge\xi)=\int_{Y}\alpha_{+b}\wedge\xi.
\end{equation}
Hence $I^{p,n-1}_+\subset K^{p,n-1}_+.$ If $\alpha_+\in K^{p,n-1}_+,$ then
$\alpha_+$ has a closed extension to $X_-.$ This follows from Theorem 5.3.1
in~\cite{FollandKohn1} and establishes~\eqref{eq5}. 
\end{proof}

We now identify $\cH^{p,n}_{\cR_-}(X_-).$
\begin{proposition}\label{prop3}  With $X, X_+, X_-$ as above, we have the isomorphism
\begin{equation}
\cH_{\cR_-}^{p,n}(X_-)\simeq \cH^{p,n}(X)\oplus\frac{ H^{p,n-1}(\bX_+)}{K^{p,n-1}_+}.
\label{eq55}
\end{equation}
\end{proposition}
\begin{remark}
If $X_+$ is a Stein manifold then the groups $H^{p,q}(X_+)$ vanish for $q>0,$
as do the groups $H^{p,q}_b(Y)$ for $1<q<n-1.$ This proposition and
Theorem~\ref{thm3}, then imply that
\begin{equation}
H^{p,q}(X)\simeq\cH^{p,q}_{\cR_-}(X_-)
\end{equation}
for all $0\leq p,q\leq n.$
\end{remark}
\begin{proof} The group $\cH^{p,n}_{\cR_-}(X_-)$ 
consists of $(p,n)$-forms $\alpha_-$ on $X_-$ that satisfy:
\begin{equation}
\dbar^*\alpha_-=0\text{ and }\bcS_p(\dbnc\alpha_-)_b=(\dbnc\alpha_-)_b.
\end{equation}
It is a simple matter to show that the first condition implies the
second. Hence if $\beta_-\in H^{n-p,0}(X_-),$ then $\dbar^*{}^{\star}\beta_-=0$
and therefore ${}^{\star}\beta_-\in \cH_{\cR_-}^{p,n}(X_-).$  From this we conclude that
the inclusion of $\cH^{p,n}(X)$  into $\cH^{p,n}_{\cR_-}(X_-)$ is
injective. The range consists of exactly those forms $\alpha_-$ such that
${}^{\star}\alpha_-$ has a holomorphic extension to  $X_+.$  Again applying Theorem
5.3.1 of~\cite{FollandKohn1}, we see that the obstruction to having such an
extension is precisely $\frac{ H^{p,n-1}(X_+)}{K^{p,n-1}_+},$ thus proving the
proposition. 
\end{proof}

Putting together this proposition with Theorem~\ref{thm3} and the results of
Section~\ref{s.nullsp} gives our first gluing formula.
\begin{corollary}\label{cor1} Suppose that $X, X_+, X_-$ are as above, then,
  for $0\leq p\leq n,$ we have the
  following identities
\begin{equation}
\begin{split}
\chi^p_{\cO}(X)=\sum_{q=0}^{n}\dim & H^{p,q}(X)(-1)^q=\\
&\sum_{q=0}^{n}[\dim \cH^{p,q}_{\cR_+}(X_+)+
\dim \cH^{p,q}_{\cR_-}(X_-)](-1)^q-\sum_{q=1}^{n-2}(-1)^q\dim H^{p,q}_b(Y).
\end{split}
\label{eq77}
\end{equation}
The last term is absent if $\dim X=2.$
\end{corollary}
\begin{proof}
The identity in~\eqref{eq77} follows from the fact that the alternating sum of
the dimensions in a long exact sequence is zero along with the consequence of
Proposition~\ref{prop3}:
\begin{equation}
\dim \cH_{\cR_-}^{p,n}(X_-)=\dim H^{p,n}(X)+\dim
\cH_{\cR_+}^{p,n-1}(\bX_+)-\dim K^{p,n-1}_+. 
\end{equation}
We also use that
\begin{equation}
\begin{split}
H^{0,0}(X)\simeq \cH^{0,0}_{\cR_-}(X_-)&\text{ and }
\cH^{p,0}_{\cR_+}(X_+)=0\text{ for all }p\geq 0\\
\cH^{p,1}_{\cR_+}(X_+)\simeq &\cH^{p,1}_{\dbar}(X_+)\oplus E^{p,1}_+
\simeq H^{p,1}(\bX_+)\oplus E^{p,1}_+.
\end{split}
\end{equation}
\end{proof}

We modify a second exact sequence in~\cite{AnHi} in order to obtain an
expression for $\chi_{\cO}^p(X)$ in terms of $\cH^{p,q}_{\cR_+}(X_+)$ and
$\cH^{p,q}_{\Id-\cR_+}(X_-).$ This formula is a subelliptic analogue of
Bojarski's formula expressing the index of a Dirac operator on a partitioned
manifold in terms of the indices of boundary value problems  on the
pieces. First we state the modification of the exact sequence from Proposition 4.3
in~\cite{AnHi}.
\begin{theorem} Let $X,X_+, X_-, Y$ be as above. Then the following sequence
 is exact
\begin{equation}
\begin{CD}
0 @>>> \cH^{p,1}_{\Id-\cR_+}(X_-) @>\talpha_1>>  H^{p,1}(\bX_-) \\
@>\beta_1>> H^{p,1}_b(Y)@>\gamma_1>> H^{p,2}(X_-,\cI)@>\alpha_2>> H^{p,2}(\bX_-) \\
@>\beta_2>>\cdots & \cdots@>\alpha_{n-2}>> H^{p,n-2}(\bX_-)\\
 @>\beta_{n-2}>>  H^{p,n-2}_b(Y) @>\gamma_{n-2}>>
H^{p,n-1}(X_-,\cI)\\@>H^{p,n-1}_{\cR_-}>>
\cH^{p,n-1}_{\cR_-}(X_-)
@>>> 0.
\end{CD}
\label{les2}
\end{equation} 
The map $\gamma_q$ is defined here by following the map $\gamma_q,$ defined above,
by restriction to $X_-.$
\end{theorem}
\begin{remark} If $n=2,$ then this sequence degenerates to 
\begin{equation}
\begin{CD}
0 @>>> \cH^{p,1}_{\Id-\cR_+}(X_-) @>H^{p,1}_{\cR_-}>>
\cH^{p,1}_{\cR_-}(X_-)
@>>> 0.
\end{CD}
\end{equation} 
In this case $H^{p,1}(\bX_-)$ is not isomorphic to $\cH^{p,1}_{\cR_-}(X_-),$
nor is $H^{p,1}(X_-,\cI)$ isomorphic to $\cH^{p,1}_{\Id-\cR_+}(X_-).$ The
argument given below shows that $H^{p,1}_{\cR_-}$ is injective for all $p.$ The
duality argument used at the end of the proof allows us to use the injectivity
of $H^{2-p,1}_{\cR_-}$ to deduce that it is also surjective.
\end{remark}
\begin{proof} We first need to show that $\cH^{p,1}_{\Id-\cR_+}(X_-)$ injects
  into $H^{p,1}(\bX_-).$ A form
  $\alpha$ belongs to $\cH^{p,1}_{\Id-\cR_+}(X_-)$ provided that
$\dbar\alpha=\dbar^*\alpha=0,$  $\alpha_b=0,$ and $\cS_p(\dbnc\alpha)_b=0.$ As
  $H^{p,1}(\bX_-)\simeq\cH^{p,1}_{\cR_-}(X_-),$ it suffices to show that
  $H^{p,1}_{\cR_-}(\alpha)=0$ if and only if $\alpha=0.$ A form in
  $\cH^{p,1}_{\Id-\cR_+}(X_-)$ belongs to $\Dom_{L^2}(\dbar^{p,1}_{\cR_-}),$
  hence, if $H^{p,1}_{\cR_-}(\alpha)=0,$ then
\begin{equation}
\alpha=\dbar\dbar^*G^{p,1}_{\cR_-}(\alpha)=\dbar\beta.
\end{equation}
Observe that $0=\alpha_b=\dbarb\beta_b.$ We can now show that $\alpha=0:$
\begin{equation}
\begin{split}
\langle \alpha,\alpha\rangle_{X_-} &=\langle \dbar\beta,\alpha\rangle_{X_-}\\
&=\langle (\dbnc\alpha)_b,\beta\rangle_Y.
\end{split}
\end{equation}
On the one hand $\cS_p(\dbnc\alpha)_b=0,$ while, on the other hand
$\cS_p(\beta_b)=\beta_b.$ This shows that $\langle
\alpha,\alpha\rangle_{X_-}=0.$ 

Now we show that $\Im\talpha_1=\Ker\beta_1.$ The containment
$\Im\talpha_1\subset\Ker\beta_1$ is clear because $\alpha_b=0$ for
$\alpha\in\cH^{p,1}_{\Id-\cR_+}(X_-).$ If $\xi\in\Ker\beta_1,$ then there is a
$(p,0)$-form, $\psi$ on $Y$ so that
\begin{equation}
\dbarb\psi=\xi_b.
\end{equation}
Let $\Psi_0$ denote a smooth extension of $\xi$ to $X_-;$ the form
$\xi-\dbar\Psi_0$ satisfies $(\xi-\dbar\Psi_0)_b=0,$ and therefore belongs to
$\Dom_{L^2}(\dbar^{p,1}_{\Id-\cR_+}).$ Hence we have the expression
\begin{equation}
\xi-\dbar\Psi_0=H^{p,1}_{\Id-\cR_+}(\xi-\dbar\Psi_0)+
\dbar\dbar^*G^{p,1}_{\Id-\cR_+}(\xi-\dbar\Psi_0). 
\end{equation}
If we let $\Psi_1=\dbar^*G^{p,1}_{\Id-\cR_+}(\xi-\dbar\Psi_0),$ then
\begin{equation}
\xi-\dbar(\Psi_0+\Psi_1)=H^{p,1}_{\Id-\cR_+}(\xi-\dbar\Psi_0).
\end{equation}
As $\xi-\dbar(\Psi_0+\Psi_1)$ and $\xi$ represent the same class 
${\xi}\in H^{p,1}(\bX_-),$ we see that $[\xi]\in\Im\talpha_1.$ This shows the exactness
at $H^{p,1}(\bX_-).$ The exactness through $H^{p,n-2}_b(Y)$ follows from
Proposition 4.3 in~\cite{AnHi}.

The next case we need to consider is $H^{p,n-1}(X_-,\cI).$  The range of
$\gamma_{n-2}$ consists of equivalence classes of exact $(p,n-1)$-forms,
$\dbar\txi,$ such that $\dbarb\xi_b=0.$ Such a form is evidently in
$\Dom_{L^2}(\dbar^{p,n-1}_{\cR_-}),$  and therefore
$H^{p,n-1}_{\cR_-}(\dbar\txi)=0.$ Now suppose that $H^{p,n-1}_{\cR_-}(\xi)=0,$
for a $\xi$ with $\dbar\xi=\xi_b=0.$ As
$\xi\in\Dom_{L^2}(\dbar^{p,n-1}_{\cR_-})$ it follows that
\begin{equation}
\xi=\dbar\dbar^*G^{p,n-1}_{\cR_-}(\xi).
\end{equation}
If we let $\theta=\dbar^*G^{p,n-1}_{\cR_-}(\xi),$ then clearly
\begin{equation}
0=\xi_b=\dbarb\theta_b,
\end{equation}
and therefore $\xi\in\Im\gamma_{n-2}.$

To complete the proof of this theorem, we need to show that $H^{p,n-1}_{\cR_-}$
is surjective. We use the isomorphism $H^{p,n-1}(X_-,\cI)\simeq
\cH^{p,n-1}_{\Id-\cR_+}(X_-).$ If
$\xi\in\cH^{p,n-1}_{\cR_-}(X_-)$ and $\theta\in\cH^{p,n-1}_{\Id-\cR_+}(X_-),$
then
\begin{equation}
\langle\xi,\theta\rangle_{X_-}=
\langle\xi, H^{p,n-1}_{\cR_-}\theta\rangle_{X_-}=
\langle H^{p,n-1}_{\Id-\cR_+}\xi,\theta\rangle_{X_-}.
\label{06.7.1}
\end{equation}
Using the relations in~\eqref{06.7.1} we see, by duality, that
$H^{p,n-1}_{\cR_-}$ is surjective if and only if $H^{p,n-1}_{\Id-\cR_+}$ is
injective. As $H^{p,n-1}_{\Id-\cR_+}={}^\star H^{n-p,1}_{\cR_-}{}^\star,$ this
injectivity follows from the proof of exactness at $\cH_{\Id-\cR_+}^{r,1}(X_-)$
for the case $r=n-p.$
\end{proof}

 We get a second gluing formula for $\chi^p_{\cO}(X).$
\begin{corollary} Suppose that $X, X_+, X_-$ are as above, then for $0\leq
  p\leq n,$ we have the following identities
\begin{equation}
\sum_{q=0}^{n}\dim  H^{p,q}(X)(-1)^q=
\sum_{q=0}^{n}[\dim \cH^{p,q}_{\cR_+}(X_+)+
\dim \cH^{p,q}_{\Id-\cR_+}(X_-)](-1)^q,
\label{eq77b}
\end{equation}
that is
\begin{equation}
\Ind(\eth^{\even}_X)=\Ind(\eth^{\even}_{+},\cR^{\even}_+)+\Ind(\eth^{\even}_{-},\Id-\cR^{\even}_+).
\label{8.2.1}
\end{equation}
\end{corollary}
\begin{proof}
These formul\ae\ follow from those in Corollary~\ref{cor1} using the
consequence of the previous theorem that
\begin{equation}
\sum_{q=1}^{n-1}\dim\cH^{p,q}_{\Id-\cR_+}(X_-)(-1)^q=
\sum_{q=1}^{n-1}\dim\cH^{p,q}_{\cR_-}(X_-)(-1)^q+
\sum_{q=1}^{n-2}\dim H^{p,q}_{b}(Y)(-1)^q.
\end{equation}
If $n=2$ the last sum is absent. To complete the proof we use the isomorphisms
\begin{equation}
\begin{split}
&\cH^{p,0}_{\cR_-}(X_-)=\cH^{p,0}_{\Id-\cR_+}(X_-)= H^{p,0}(X_-)\\
&\cH^{p,n}_{\cR_-}(X_-)=\cH^{p,n}_{\Id-\cR_+}(X_-)\simeq
  [H^{n-p,0}(X_-)]^{\star}.
\end{split}
\end{equation}
\end{proof}

\begin{remark} These formul\ae\ are
    exactly what would be predicted, in the elliptic case, from Bojarski's
    formula: Let $\cP_{\pm}^{\eo}$ denote the Calderon projectors for
    $\dbar+\dbar^*$ acting on $\Lambda^{p,\eo} X_{\pm}.$ Bojarski proved that,
\begin{equation}
\Ind(\eth^{\even}_{X})=\Rind(\Id-\cP^{\even}_{-},\cP^{\even}_{+}).
\end{equation}
Let $P$ be a projection in the Grassmanian of $\cP^{\even}_{+}.$ From Bojarski's formula we
 easily deduce the following identity
\begin{equation}
\Ind(\eth^{\even}_{X})=\Ind(\eth^{\even}_+,P)+\Ind(\eth^{\even}_-,\Id-P).
\label{eqn6.15.1}
\end{equation}
The proof uses elementary properties of the relative index:
\begin{equation}
\begin{split}
-\Rind(P_2,P_1)=
&\Rind(P_1,P_2)=-\Rind(\Id-P_1,\Id-P_2)\\
\Rind(P_1,P_3)&=\Rind(P_1,P_2)+\Rind(P_2,P_3).
\end{split}
\label{eqn6.15.2}
\end{equation}
To deduce~\eqref{eqn6.15.1} we use the observation that
\begin{equation}
 \Ind(\eth^{\even}_{+},P)=\Rind(\cP^{\even}_{+},P),\quad\Ind(\eth^{\even}_{-},\Id-P)=\Rind(\cP^{\even}_{-},\Id-P).
\end{equation}
Hence, we see that
\begin{equation}
\begin{split}
\Ind(\eth^{\even}_{+},P)+\Ind(\eth^{\even}_{-},\Id-P)&=\Rind(\cP^{\even}_{+},P)+\Rind(\cP^{\even}_{-},\Id-P)\\
&=\Rind(\cP^{\even}_{+},P)-\Rind(\Id-\cP^{\even}_{-},P)\\
&=\Rind(\cP^{\even}_{+},\Id-\cP^{\even}_{-}).
\end{split}
\end{equation}
 The proofs of the identities in~\eqref{eqn6.15.2} use the theory of Fredholm
 pairs. If $H$ is a Hilbert space, then a pair of subspaces $H_1, H_2$ of $H$
 is a Fredholm pair if $H_1\cap H_2$ is finite dimensional, $H_1+H_2$ is closed
 and $H/(H_1+H_2)\simeq H_1^{\bot}\cap H_2^{\bot}$ is finite dimensional. One
 uses that, for two admissible projectors $P_1,P_2,$ the subspaces of
 $L^2(Y;E)$ given by $H_1=\Im P_1, H_2=\Im(\Id-P_2)$ are a Fredholm pair and
\begin{equation}
\Rind(P_1,P_2)=\dim H_1\cap H_2-\dim H_1^{\bot}\cap H_2^{\bot}. 
\end{equation}

In our case the projectors are $\cP^{\even}_{\pm}$ and $\cR^{\even}_{\pm}.$
While it is true that, e.g. $\Im \cP^{\even}_{+}\cap\Im (\Id-\cR^{\even}_{+})$
is finite dimensional, it is not true that $\Im \cP^{\even}_{+}+\Im
(\Id-\cR^{\even}_{+})$ is a closed subspace of $L^2.$ So these projectors do
not define a traditional Fredholm pair. If we instead consider these operators
as acting on smooth forms, then the $\Im \cP^{\even}_{+}$ and $\Im
(\Id-\cR^{\even}_{+})$ are a ``Frechet'' Fredholm pair. As the result predicted by
Bojarski's theorem remains true, this indicates that perhaps there is a
generalization of the theory of Fredholm pairs that includes both the elliptic
and subelliptic cases.

It seems a natural question whether the Agranovich-Dynin formula holds on the
pseudoconcave side as well, that is
\begin{equation}
\Ind(\eth^{\even}_-,\Id-\cR^{\prime\even}_+)+
\Ind(\eth^{\even}_-,\Id-\cR^{\even}_+)\overset{?}{=}\Rind(\cS'_p,\cS_p).
\end{equation}
If this were the case, then~\eqref{8.2.1} would also hold for boundary
conditions defined by generalized Szeg\H o projectors. Because the null space
of $(\eth^{\even}_-,\Id-\cR^{\prime\even}_+)$ does not seem to split as a
direct sum over form degrees, the argument used to prove Theorem~\ref{thm3.3}
does not directly apply to this case.
\end{remark}

\section{General holomorphic coefficients}
Thus far we have considered the Dirac operator acting on sections of
$\Lambda^{p,\eo}.$ Essentially everything we have proved for cases where $p>0$
remains true if the bundles $\Lambda^{p,\eo}$ are replaced by
$\Lambda^{\eo}\otimes\cV,$ where $\cV\to X$ is a holomorphic vector
bundle. In~\cite{Epstein3} we prove the necessary estimates for the twisted
Dirac operator acting on sections of $\Lambda^{\eo}\otimes\cV.$ For example,
suppose that $X_+$ is strictly pseudoconvex, then defining
\begin{equation}
E^{\cV,1}_{0}(\bX_+)=\frac{\{\dbar\alpha:\: \alpha\in\CI(\bX_+;\cV)
\text{ and }\dbar_b\alpha_b=0\}}
{\{\dbar\alpha:\: \alpha\in\CI(\bX_+;\cV)
\text{ and }\alpha_b=0\}},
\end{equation}
we can easily show that
\begin{equation}
\Ind(\eth^{\even}_{\cV+},\cR^{\even}_{+})=-\dim E^{\cV,1}_{0}+\sum_{q=1}^n
H^{q}(X_+;\cV).
\end{equation}
The vector space $E^{\cV,1}_{0}$ is the obstruction to extending
$\dbarb$-closed sections of $\cV\restrictedto_{bX_+}$ as holomorphic sections
of $\cV.$ Hence it is isomorphic to
$H^{n-1}_{\dbar}(X_+;\Lambda^{n,0}\otimes\cV'),$ see Proposition 5.13
in~\cite{kohn-rossi}. It is therefore finite dimensional, and vanishes if $X_+$
is a Stein manifold.

The Agranovich-Dynin formula and the Bojarski formula also hold for general
holomorphic coefficients.
\begin{theorem}
Let $X_+$ be a compact strictly pseudoconvex K\"ahler manifold and $\cV\to X_+$
a holomorphic vector bundle.  If the classical Szeg\H o projector onto the null
space of $\dbarb,$ acting on sections of $\cV\restrictedto_{bX_+}$ is denoted
$\cS_{\cV},$ and $\cS_{\cV}'$ is a generalized Szeg\H o projector, then
\begin{equation}
\Ind(\eth^{\even}_{\cV+},\cR^{\prime\even}_+)-
\Ind(\eth^{\even}_{\cV+},\cR_+^{\even})=\Rind(\cS_{\cV},\cS_{\cV}').
\end{equation}
\end{theorem}

\begin{corollary} Suppose that $X, X_+, X_-$ are as above and $\cV\to X$ is a
  holomorphic vector bundle, then  we have the following identity
\begin{equation}
\sum_{q=0}^{n}\dim
H^{q}(X;\cV)(-1)^q=\sum_{q=0}^{n}[\dim \cH^{q}_{\cR_+}(X_+;\cV)+
\dim \cH^{q}_{\Id-\cR_+}(X_-;\cV)](-1)^q
\end{equation}
that is
\begin{equation}
\Ind(\eth^{\even}_{\cV+})=\Ind(\eth^{\even}_{\cV+},\cR^{\even}_{+})
+\Ind(\eth^{\even}_{\cV-},\Id-\cR^{\even}_{+}).
\label{eq777b}
\end{equation}
\end{corollary}

The proofs of these statements are essentially identical to those given above
and are left to the interested reader.

\end{document}